\documentclass[a4paper]{article}
\title{Cyclic $A_{\infty}$-algebras and double Poisson algebras}
\author{David Fern\'andez and Estanislao Herscovich
}
\date{}

\usepackage{amsmath,amsthm,amsfonts,amssymb,stmaryrd}
\usepackage{enumitem}
\usepackage{latexsym}
\usepackage{fullpage}
\usepackage{titlesec}
\usepackage{everypage}
\usepackage{mathrsfs}
\usepackage{a4,xcolor,palatino,fancyhdr}
\colorlet{mycyan}{cyan!40!gray}
\colorlet{myblue}{blue!40!gray}
\colorlet{myred}{red!40!gray}
\colorlet{mygreen}{green!20!gray}
\definecolor{ultramarine}{RGB}{0,32,96}
\definecolor{light-gray}{gray}{0.8}
\colorlet{myultramarine}{ultramarine!20!gray}
\usepackage{graphicx}
\usepackage{float}  
\usepackage[small, bf, margin=90pt, tableposition=bottom]{caption}
\usepackage[utf8]{inputenc}
\usepackage[T1]{fontenc}
\usepackage[english,french]{babel}
\usepackage[colorlinks=true,citecolor=mygreen, linkcolor=mygreen, urlcolor=myultramarine,breaklinks, naturalnames]{hyperref}
\usepackage[backrefs]{amsrefs}

\input{xy}
\xyoption{all}
\xyoption{poly}
\usepackage[all]{xy}

\newtheorem{theorem}{Theorem}[section]
\newtheorem{proposition}[theorem]{Proposition}
\newtheorem{definition}[theorem]{Definition}
\newtheorem{lemma}[theorem]{Lemma}
\newtheorem{corollary}[theorem]{Corollary}
\newtheorem{remark}[theorem]{Remark}

\newtheorem{fact}[theorem]{Fact}

\newtheorem{convention}[theorem]{Convention}

\DeclareFontFamily{U}{BOONDOX-calo}{\skewchar\font=45 }
\DeclareFontShape{U}{BOONDOX-calo}{m}{n}{
  <-> s*[1.05] BOONDOX-r-calo}{}
\DeclareFontShape{U}{BOONDOX-calo}{b}{n}{
  <-> s*[1.05] BOONDOX-b-calo}{}
\DeclareMathAlphabet{\mathcalboondox}{U}{BOONDOX-calo}{m}{n}
\SetMathAlphabet{\mathcalboondox}{bold}{U}{BOONDOX-calo}{b}{n}
\DeclareMathAlphabet{\mathbcalboondox}{U}{BOONDOX-calo}{b}{n}

\newcommand{\lr}[1]{
  \{\mkern-6mu\{#1\}\mkern-6mu\}}

\numberwithin{equation}{section}

\newcommand{\gan}{\mathchoice
  {\hskip -0.5ex\raise-0.5ex\hbox{\includegraphics[height = 2.5ex, width = 2.5ex]{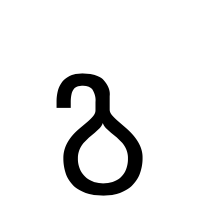}}\hskip -0.5ex}
  {\hskip -0.5ex\raise-0.5ex\hbox{\includegraphics[height = 2.5ex, width = 2.5ex]{gan.png}}\hskip -0.5ex}
  {\hskip -0.2ex\raise-0.2ex\hbox{\includegraphics[height = 1.5ex, width = 1.5ex]{gan.png}}\hskip -0.2ex}
  {\hskip -0.1ex\vcenter{\hbox{\includegraphics[height = 1ex, width = 1ex]{gan.png}}}\hskip -0.1ex}
}

\newcommand{\xan}{\mathchoice
  {\hskip -0.5ex\raise-0.5ex\hbox{\includegraphics[height = 2.5ex, width = 2.5ex]{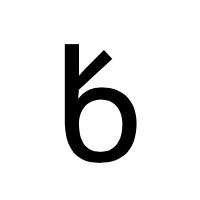}}\hskip -0.5ex}
  {\hskip -0.5ex\raise-0.5ex\hbox{\includegraphics[height = 2.5ex, width = 2.5ex]{xan.png}}\hskip -0.5ex}
  {\hskip -0.2ex\raise-0.2ex\hbox{\includegraphics[height = 1.5ex, width = 1.5ex]{xan.png}}\hskip -0.2ex}
  {\hskip -0.1ex\vcenter{\hbox{\includegraphics[height = 1ex, width = 1ex]{xan.png}}}\hskip -0.1ex}
}

\def\place{{-}}

\newcommand\ZZ{{\mathbb{Z}}}
\newcommand\NN{{\mathbb{N}}}
\newcommand{\kk}{{\Bbbk}} 
\newcommand{\D}{\operatorname{\mathbb{D}er}}
\newcommand{\Der}{\operatorname{Der}}

\def\place{{-}}

\titlespacing*{\section}
{0pt}{1.25ex plus 1ex minus .2ex}{1.25ex plus .2ex}

\begin{document}

\maketitle
                                                     
\hrulefill
\selectlanguage{english}
\begin{abstract} 
In this article we prove that there exists an explicit bijection between nice $d$-pre-Calabi-Yau algebras and $d$-double Poisson differential graded algebras, where $d \in \mathbb{Z}$, extending a result proved by N. Iyudu and M. Kontsevich. We also show that this correspondence is functorial in a quite satisfactory way, giving rise to a (partial) functor from the category of $d$-double Poisson dg algebras to the partial category of $d$-pre-Calabi-Yau algebras. 
Finally, we further generalize it to include double $P_{\infty}$-algebras, as introduced by T. Schedler. 
\end{abstract}



\hrulefill
\section{Introduction}

Pre-Calabi-Yau algebras were introduced in \cite{KV18}, and further studied in \cite{I17} and \cite{IK17}.
However, these structures (or equivalent ones) have appeared in other works under different names, such as \emph{$V_{\infty}$-algebras} in \cite{TZ16}, \emph{$A_{\infty}$-algebras with boundary} in \cite{Sei12}, \emph{noncommutative divisors} in Remark 2.11 in \cite{Sei17}, or \emph{weak Calabi-Yau structures} (see \cite{Ko13} for the case of algebras, \cite{Ye18} for differential graded (dg) categories and \cite{KPS17} for linear $\infty$-categories).
These references show that pre-Calabi-Yau structures play an important role in homological algebra, symplectic geometry, string topology, noncommutative geometry and even in Topological Quantum Field Theory (see \cite{Ko13}). 
Following \cite{KoSo09}, a (compact) \emph{Calabi-Yau structure} (of dimension $n$) on a compact $A_{\infty}$-algebra $A$ is a nondegenerate cyclically invariant pairing on $A$ of degree $n$. 
In the sense of formal noncommutative geometry, it is the analogue of a symplectic structure.
The problem with this definition is that for applications related to path spaces, Fukaya categories, open Calabi-Yau manifolds or Fano manifolds, the hypothesis of compactness is too restrictive. 
This was the reason why pre-Calabi-Yau algebras were originally introduced in \cite{KV18}.

Roughly speaking, a pre-Calabi-Yau algebra can be regarded as a formal noncommutative Poisson structure on a non-compact algebra because it is a noncommutative analogue of a solution to the Maurer-Cartan equation for the Schouten bracket on polyvector fields. 
More precisely, let $A$ be a $\mathbb{Z}$-graded vector space, and let $C^{(k)}(A):=\prod_{r\geq 0}\mathcal{H}om(A[1]^{\otimes r},A^{\otimes k})$, for $k\geq 1$. 
A \emph{pre-Calabi-Yau structure} on $A$ is a solution $m=\sum_{k\geq 0}m^{(k)}$, $m^{(k)}\in C^{(k)}(A)$ of the Maurer-Cartan equation $[m,m]_{\text{gen.neckl}}=0$ (see \cite{IK17}, Def. 2.5). 
Here, $[\,,\,]_{\text{gen.neckl}}$ is the ``generalized necklace bracket", which is a kind of graded commutator 
(see \cite{IK17}, Def. 2.4). 
Nevertheless, for our purposes, we will use a different but equivalent version of this notion (see \cite{IK17}, Prop. 2.7).  
A pre-Calabi-Yau algebra essentially is a cyclic $A_{\infty}$-algebra structure on $A \oplus A^{\#}[d-1]$ for the natural bilinear form of degree $d-1$ induced by evaluation such that $A$ is an $A_{\infty}$-subalgebra (see Definition \ref{definition-pre-d-CY}). 

If pre-Calabi-Yau structures are regarded as noncommutative Poisson structures in the setting of formal noncommutative geometry, double Poisson algebras are the natural candidates for Poisson structures in the context of noncommutative differential geometry based on double derivations as developed in \cite{CBEG} and \cite{vdB}. 
Indeed, let $\D A =\Der(A,A\otimes A)$ be the $A$-bimodule of double derivations, and let $DA =T_A(\D A)$ be its tensor algebra. 
Roughly speaking, a double Poisson algebra is an algebra endowed with a bivector $P\in (DA)_2$ such that $\{P,P\}=0$, where $\{ \hskip 0.6mm , \}$ is a kind of commutator in this context (see \cite{vdB}, Section 4.4). 
Besides their similarity with the commutative notion, double Poisson algebras turn out to be the appropriate noncommutative Poisson algebras in this setting because they satisfy the Kontsevich-Rosenberg principle (see \cite{KoRo} and \cite{vdB}, Section 7.5), whereby a structure on an associative algebra has geometric meaning if it induces standard geometric structures
on its representation spaces.

Hence, since pre-Calabi-Yau algebras and double Poisson algebras can be regarded as noncommutative Poisson structures, one should expect some relationship between them. 
For instance, W.-K. Yeung \cite{Ye18} proved that double Poisson structures on dg categories provide examples of pre-Calabi-Yau structures. 
Furthermore, given an associative algebra $A$, N. Iyudu and M. Kontsevich showed that there exists an explicit one-to-one correspondence between the class of non-graded double Poisson algebras and that of 
pre-Calabi-Yau algebras whose multiplications $m_{i}$ vanish for $i \in \NN \setminus \{ 2,3 \}$, such that $m_{2}$ is the usual product of the square-zero extension $A \oplus A^{\#}[d-1]$, and $m_3$ sends $A\otimes A^{\#}\otimes A$ to $A$ 
and $A^{\#}\otimes A\otimes A^{\#}$ to $A^{\#}$ (see \cite{IK17}, Thm. 1.1). 

The first main result of this article is an extension of this correspondence to the differential graded setting (see Theorem \ref{theorem:main1}).
Our second main result shows that such a correspondence satisfies a simple functorial property (see Theorems \ref{theorem:main2} and \ref{theorem:main2bis}), for a suitable notion of \emph{morphism of $d$-pre-Calabi-Yau algebras} (Definition \ref{definition:morprecy}). 
We remark that this notion does not define a category but a \emph{partial category} of $d$-pre-Calabi-Yau algebras, 
since not all pairs $(f,g)$ of morphisms such that the codomain of $f$ is the domain of $g$ are composable. 

Moreover, T. Schedler \cite{Sch09} showed an interesting connection of the classical and associative Yang-Baxter equations with double Poisson algebras, that he generalized to $L_{\infty}$-algebras, giving rise to ``infinity'' versions of Yang-Baxter equations and double Poisson algebras.
The latter arise by relaxing the (double) Jacobi identity up to homotopies, but not the associativity of the multiplication. 
We recall Schedler's definition of \emph{double $P_{\infty}$-algebras} in Definition \ref{definition:poinf}, 
which coincides with the usual notion of dg double Poisson algebras if the higher brackets vanish. 
The third main result of the article states that there is also a correspondence between certain 
pre-CY structures on (nonunitary) graded algebras $A$ and double $P_{\infty}$-algebras, giving a different extension of Theorem \ref{theorem:main1} if $d=0$ (see Theorem \ref{theorem:main3}). 

We believe that our results can be a powerful tool to define both new double Poisson and pre-Calabi-Yau structures. 
For example, the study of linear and quadratic double Poisson brackets on free associative algebras, as in \cite{ORS13} or \cite{PVdW08}, might be useful to better understand and extend the results obtained by N. Iyudu in \cite{I17}, where pre-Calabi-Yau structures on path algebras of quivers with one vertex and a finite number of loops are studied.
Moreover, the results obtained in this article give rise to a more natural study of quasi-isomorphism classes of dg double Poisson algebras by considering the associated pre-Calabi-Yau $A_{\infty}$-algebras. 
We remark that the former problem is in principle specially difficult, as it is usually the case when dealing with double structures (\textit{e.g.} double associative algebras, double Poisson algebras), since, although transfer theorems for strongly homotopic structures over dioperads or properads are known to hold, they are not explicit. 
Indeed, as a major difference with the theory of (al)gebras over operads we can mention that there does not exist in general a Schur functor construction for dioperads/properads --so there is in particular no bar construction for (al)gebras over dioperads/properads--, the category of (al)gebras over dioperads/properads does not carry any natural model structure, etc.  

The contents of the article are as follows. 
We begin in Section \ref{section:preliminaries} by fixing our notations and conventions, 
and in Section \ref{section:Poisson} we review some known definitions and results related to double Poisson dg algebras. 
After reviewing the basic definitions and results on $A_{\infty}$-algebras in the first part of Section \ref{section:cyclic-pre-CY}, 
we recall the crucial notion of a $d$-pre-Calabi-Yau structure as well as some additional conditions on $A_{\infty}$-algebras that we will need to prove our main results. 

Section \ref{section:core} is the core of the article. 
Subsection \ref{subsection:main1} is devoted to prove the first main result of our article, Theorem \ref{theorem:main1}, that establishes the bijection between fully manageable nice $d$-pre-Calabi-Yau structures 
and double Poisson brackets of degree $-d$. 
In Subsection \ref{subsection:main2} we prove our second main result, namely the functoriality of the previous correspondence (see Theorems \ref{theorem:main2} and \ref{theorem:main2bis}). 
Finally, in Section \ref{section:P-infinity}, we prove our last main result, Theorem \ref{theorem:main3}, that extends the previous bijection in case $d=0$ to include double $P_{\infty}$-algebras.

\paragraph*{Acknowledgments.} The first author is supported by the Alexander von Humboldt Stiftung in the framework of an Alexander von Humboldt professorship endowed by the German Federal Ministry of Education and Research. 
The second author was supported by the GDRI ``Representation Theory'' 2016-2020 and the BIREP group, 
and is deeply thankful to Henning Krause and William Crawley-Boevey for their hospitality at the University of Bielefeld. 
We are very grateful to Yiannis Vlassopoulos for sharing with us the manuscript \cite{KV18}.

\section{Notations and conventions}
 \label{section:preliminaries}
\subsection{Generalities}

In what follows, $\kk$ will denote a field of characteristic zero. 
We recall that, if $V = \oplus_{n \in \ZZ} V^{n}$ is a (cohomological) graded vector space (resp., dg vector space with differential $\partial_{V}$), $V[m]$ is the graded (resp., dg) vector space over $\kk$ whose $n$-th homogeneous component $V[m]^{n}$ is given by $V^{n+m}$, for all $n, m \in \ZZ$ (resp., and whose differential $\partial_{V[m]}$ sends a homogeneous $v \in V^{n+m}$ to $(-1)^{m} \partial_{V}(v)$).
It is called the \emph{shift} of $V$. 
Given a nonzero element $v \in V^{n}$, we will denote $|v| = n$ the \emph{degree} of $v$. 
If we refer to the degree of an element, we will be implicitly assuming that it is nonzero and homogeneous. 

We recall that a \emph{morphism} $f : V \rightarrow W$ of graded (resp., dg) vector spaces of degree $d \in \ZZ$ is a homogeneous linear map of degree $d$, \textit{i.e.} $f(V^{n}) \subseteq W^{n+d}$ for all $n \in \ZZ$, (resp., satisfying that $f \circ \partial_{V} = (-1)^{d} \partial_{W} \circ f$). 
A morphism of degree zero will be called \emph{closed}. 
Moreover, if $f : V \rightarrow W$ is a morphism of graded (resp., dg) vector spaces of degree $d$, $f[m] : V[m] \rightarrow W[m]$ is the morphism of degree $d$ whose underlying set-theoretic map is $(-1)^{m d} f$. 
In this way, the shift $(\place)[m]$ defines an endofunctor on the category of graded (resp., dg) vector spaces provided 
with closed morphisms. 

Given any $d \in \ZZ$, we will denote by $s_{V}^{d} : V \rightarrow V[d]$ the \emph{suspension morphism}, whose underlying map is the identity of $V$, and $s_{V}^{1}$ will be denoted simply by $s_{V}$.
To simplify notation, we write $sv$ instead of $s_V(v)$ for a homogeneous $v\in V$.
All morphisms between vector spaces will be $\kk$-linear (satisfying further requirements if the spaces are further decorated). 
All unadorned tensor products $\otimes$ would be over $\kk$. 
Since graded vector spaces can be considered as dg vector spaces with trivial differentials, we will proceed to consider the case of dg vector spaces. 
We also remark that $\mathbb{N}$ will denote the set of positive integers, whereas $\mathbb{N}_{0}$ will be the set of 
nonnegative integers. 

\subsection{Permutations}

Given $n \in \NN$, we will denote by $\mathbb{S}_{n}$ the group of permutations of $n$ elements $\{ 1, \dots, n\}$, 
and given any $\sigma \in \mathbb{S}_{n}$, $\operatorname{sgn}(\sigma) \in \{ \pm 1\}$ will denote its sign. 
Given two dg vector spaces $V$ and $W$, we denote by $\tau_{V,W} : V \otimes W \rightarrow W \otimes V$ 
the closed morphism determined by $v \otimes w \mapsto (-1)^{|v| |w|} w \otimes v$, for all homogeneous elements $v \in V$ and $w \in W$. 
Moreover, given any transposition $\varsigma = (i j)$ with $i<j$ in the group of permutations $\mathbb{S}_{n}$ of $n \in \NN$ elements, it induces a unique closed morphism $\tau_{V,n}(\varsigma) : V^{\otimes n} \rightarrow V^{\otimes n}$, sending 
$v_{1} \otimes \dots \otimes v_{n}$ to 
\[      (-1)^{\epsilon} v_{1} \otimes \dots \otimes v_{i-1} \otimes v_{j} \otimes v_{i+1} \otimes \dots \otimes v_{j-1} \otimes v_{i} \otimes v_{j+1} \otimes \dots \otimes v_{n},     \]
where $\epsilon = |v_{i}||v_{j}| + (|v_{i}|+|v_{j}|) (\sum_{\ell =i+1}^{j-1} |v_{\ell}|)$, for all homogeneous $v_{1}, \dots, v_{n}$  in $V$.
More generally, for any permutation $\sigma \in \mathbb{S}_{n}$, written as a composition of transpositions $\varsigma_{1} \circ \dots \circ \varsigma_{m}$, we define the closed morphism 
$\tau_{V,n}(\sigma) : V^{\otimes n} \rightarrow V^{\otimes n}$ given by $\tau_{V,n}(\varsigma_{1}) \circ \dots \circ \tau_{V,n}(\varsigma_{m})$. 
We leave to the reader the verification that this is independent of the choice of the transpositions used in the decomposition of $\sigma$. 
In fact, it is easy to check that $\tau_{V,n}(\sigma)$ sends $\bar{v} = v_{1} \otimes \dots \otimes v_{n}$ to 
\begin{equation}
\label{eq:permeps1}
      (-1)^{\epsilon(\sigma,\bar{v})} v_{\sigma^{-1}(1)} \otimes \dots \otimes v_{\sigma^{-1}(n)},     
\end{equation}   
where 
\begin{equation}
\label{eq:permeps2}
\epsilon(\sigma,\bar{v}) = \underset{\text{\begin{tiny} $\begin{matrix}i<j,\\ \sigma^{-1}(i) > \sigma^{-1}(j)\end{matrix}$ \end{tiny}}}{\sum} |v_{\sigma^{-1}(i)}| |v_{\sigma^{-1}(j)}|.
\end{equation}          
We will usually write $\sigma$ instead of $\tau_{V,n}(\sigma)$ to simplify the notation. 

\subsection{The closed monoidal structure}

Given two dg vector spaces $V$ and $W$ we will denote by $\mathcal{H}om(V,W)$ the dg vector space whose 
component of degree $d$ is formed by all morphisms from $V$ to $W$ of degree $d$, and whose differential sends an homogeneous element $f \in \mathcal{H}om(V,W)$ to $\partial_{W} \circ f - (-1)^{|f|} f \circ \partial_{V}$. 
If $W = \kk$, we will denote $\mathcal{H}om(V,\kk)$ by $V^{\#}$. 
If $f : V \rightarrow V'$ is a morphism of degree $d$, then $\mathcal{H}om(f,W) : \mathcal{H}om(V',W) \rightarrow \mathcal{H}om(V,W)$ and $\mathcal{H}om(W,f) : \mathcal{H}om(W,V) \rightarrow \mathcal{H}om(W,V')$ are defined by $\mathcal{H}om(f,W)(g) = (-1)^{|f| |g|} g \circ f$ and $\mathcal{H}om(W,f)(g)= f \circ g$, respectively. 
If $W = \kk$, then $\mathcal{H}om(f,\kk)$ will be denoted by $f^{\#}$. 

It is easy to check that, given homogeneous morphisms $f : V \rightarrow V''$ and $g : V' \rightarrow V$, then  
\begin{equation}
\label{eq:Comp0}
      \mathcal{H}om(g,W) \circ \mathcal{H}om(f,W)  = (-1)^{|f| |g|} \mathcal{H}om(f \circ g,W),
\end{equation} 
and
\begin{equation}
\label{eq:Comp1}
      \mathcal{H}om(W,f) \circ \mathcal{H}om(W,g)  = \mathcal{H}om(W,f \circ g).
\end{equation} 

The usual tensor product $V \otimes W$ of vector spaces is a dg vector space 
for the grading given by $V \otimes W= \oplus_{n \in \ZZ} (V \otimes W)^{n}$, where $(V \otimes W)^{n} = \oplus_{m \in \ZZ} V^{m} \otimes W^{n-m}$, and the differential sends $v \otimes w$ to $\partial_{V}(v) \otimes w + (-1)^{|v|} v \otimes \partial_{W}(w)$, for all homogeneous $v \in V$ and $w \in W$. 
Given $f \in \mathcal{H}om(V,W)$ and $g \in \mathcal{H}om(V',W')$, the map $\Lambda_{V,V',W,W'}(f \otimes g) \in \mathcal{H}om(V \otimes V',W\otimes W')$ is the unique morphism sending $v \otimes w$ to $(-1)^{|g| |v|} f(v) \otimes g(w)$. 
This gives a closed morphism $\Lambda_{V,V',W,W'} : \mathcal{H}om(V,W) \otimes \mathcal{H}om(V',W') \rightarrow \mathcal{H}om(V \otimes V',W\otimes W')$. 
If $W = W' = \kk$, we denote it by $\lambda_{V,V'}$. 
Moreover, if it is clear from the context, we will denote $\Lambda_{V,V',W,W'}(f \otimes g)$ simply by $f \otimes g$.
Note that, using this notation, the differential of $V \otimes W$ is precisely $\partial_{V} \otimes \mathrm{id}_{W} + \mathrm{id}_{V} \otimes \partial_{W}$.  

It is easy to check that 
\begin{equation}
\label{eq:Sim0}
      \Lambda_{V,V',W,W'}(f \otimes g) \circ \Lambda_{U,U',V,V'}(f' \otimes g')  = (-1)^{|f'| |g|} \Lambda_{U,U',W,W'}\big((f \circ f') \otimes (g \circ g')\big),
\end{equation} 
and
\begin{equation}
\label{eq:Sim1}
     \lambda_{W,V} \circ \tau_{V^{\#},W^{\#}} = \tau_{W,V}^{\#}\circ \lambda_{V,W},
\end{equation} 
as well as
\begin{equation}
\label{eq:Sim2}
     \lambda_{V,W} \circ \Lambda_{U^{\#},W^{\#},V^{\#},W^{\#}}(h^{\#} \otimes \mathrm{id}_{W^{\#}}) = \big(\Lambda_{V,W,U,W}(fh\otimes \mathrm{id}_{W})\big)^{\#} \circ \lambda_{U,W},
\end{equation} 
for any homogeneous morphism $h : V \rightarrow U$. 

For later use, we recall that, given $v_{1}, \dots, v_{n} \in V$ homogeneous elements of a graded vector space, and 
$f_{1}, \dots, f_{n} \in V^{\#}$ homogeneous elements, then 
\begin{equation}
\label{eq:unipermvecfun}
   (f_{1} \otimes \dots \otimes f_{n})\big(\sigma (v_{1} \otimes \dots\otimes v_{n})\big) = \big(\sigma^{-1}(f_{1} \otimes \dots \otimes f_{n})\big)(v_{1} \otimes\dots\otimes v_{n}).
\end{equation} 

\subsection{The closed monoidal structure and the suspension}

Given $d \in \ZZ$, and $V$ and $W$ two dg vector spaces, define the closed isomorphisms 
$\mathcal{L}_{V,W}^{d} : \mathcal{H}om(V,W)[d] \rightarrow \mathcal{H}om(V[-d],W)$ and $\mathcal{R}_{V,W}^{d} : \mathcal{H}om(V,W)[d] \rightarrow \mathcal{H}om(V,W[d])$ 
given by $s^{d}_{\mathcal{H}om(V,W)}f \mapsto (-1)^{d |f|} f \circ s^{d}_{V[-d]}$ and $s^{d}_{\mathcal{H}om(V,W)}f \mapsto s^{d}_{W} \circ f$, respectively. 

Moreover, define also the closed isomorphisms 
$\mathscr{L}_{V,W}^{d} : (V \otimes W)[d] \rightarrow (V[d]) \otimes W$ and $\mathscr{R}_{V,W}^{d} : (V \otimes W)[d] \rightarrow V \otimes (W[d])$ 
given by $s^{d}_{V \otimes W}(v \otimes w) \mapsto s^{d}_{V}(v) \otimes w$ and $s^{d}_{V \otimes W}(v \otimes w) \mapsto (-1)^{d |v|} v \otimes s^{d}_{W}(w)$, respectively. 

\section{Double Poisson brackets on dg algebras}
\label{section:Poisson}

The definitions of dg algebras (possibly with unit) and dg (bi)modules are supposed to be well-known. 
We will recall however the definition of a double Poisson dg algebra, specially to avoid some imprecisions concerning signs that exist in the literature.   
The reader might check that the definition coincides with the one introduced in \cite{vdB}, Section 2.7, for the case where the differential vanishes. 
\begin{definition}
\label{definition:po}
Let $(A,\mu_{A},\partial_{A})$ be a dg algebra and $d \in \ZZ$. 
A double Poisson bracket on $A$ of degree $-d$ is a closed morphism of dg vector spaces
\[     \lr{\hskip 0.4mm,}_{A} : A[d] \otimes A[d] \longrightarrow A \otimes A     \]
of degree $d$ satisfying that 
\begin{enumerate}[label={\textup{(\roman*)}}]
\setcounter{enumi}{0} 
\item\label{item:dpa1} $- \lr{\hskip 0.4mm,}_{A} \circ \tau_{A[d],A[d]} = \tau_{A,A} \circ \lr{\hskip 0.4mm,}_{A}$;
\item\label{item:dpa2} for any $a \in A$, the homogeneous map $\operatorname{AD}(a) : A \rightarrow A \otimes A$ of degree $|a|-d$ given by $b \mapsto \lr{s_{A}^{d}a,s_{A}^{d}b}_{A}$ is a \emph{double derivation} of $A$, \textit{i.e.}
\[     \operatorname{AD}(a) \circ \mu_{A} = (\mathrm{id}_{A} \otimes \mu_{A}) \circ (\operatorname{AD}(a) \otimes \mathrm{id}_{A}) + (\mu_{A} \otimes \mathrm{id}_{A}) \circ (\mathrm{id}_{A}  \otimes \operatorname{AD}(a));     \]
\item\label{item:dpa3} $\sum_{\sigma \in C_{3}} \tau_{A,3}(\sigma) \circ \lr{\hskip 0.4mm, \hskip 0.4mm,}_{A,L} \circ \tau_{A[d],3}(\sigma^{-1}) = 0$;
\end{enumerate}
where $C_{3} \subseteq \mathbb{S}_{3}$ is the subgroup of cyclic permutations, and $\lr{\hskip 0.4mm, \hskip 0.4mm,}_{A,L} : A[d]^{\otimes 3} \rightarrow A^{\otimes 3}$ is the map 
$(\lr{\hskip 0.4mm,}_{A} \otimes \mathrm{id}_{A}) \circ (\mathrm{id}_{A[d]} \otimes s_{A}^{d} \otimes \mathrm{id}_{A}) \circ (\mathrm{id}_{A[d]} \otimes \lr{\hskip 0.4mm,}_{A})$. 
\label{dg-Poisson-double}
\end{definition}

Usually, the identity in \ref{item:dpa2} is called the \emph{Leibniz property}, and \ref{item:dpa3} is the \emph{double Jacobi identity}.
\begin{remark}
\label{remark:leibalg}
Note that $\lr{\hskip 0.4mm,}_{A}$ being a closed morphism of dg vector spaces means precisely that 
\[     \big(\partial_{A} \otimes \mathrm{id}_{A} + \mathrm{id}_{A} \otimes \partial_{A}\big) \circ \lr{\hskip 0.4mm,}_{A} = (-1)^{d} \lr{\hskip 0.4mm,}_{A} \circ \big(\partial_{A[d]} \otimes \mathrm{id}_{A[d]} + \mathrm{id}_{A[d]} \otimes \partial_{A[d]}\big).     \]
On the other hand, condition \ref{item:dpa2} in the previous definition is tantamount to the following one. 
Set $\lr{\hskip 0.4mm,}_{A}^{u} : A \otimes A \rightarrow A \otimes A$ to be the map $\lr{\hskip 0.4mm,}_{A}^{u} = \lr{\hskip 0.4mm,}_{A} \circ (s_{A}^{d} \otimes s_{A}^{d})$. 
Then, condition \ref{item:dpa2} is equivalent to 
\[     \lr{\hskip 0.4mm,}_{A}^{u} \circ (\mathrm{id}_{A} \otimes \mu_{A}) = 
(\mathrm{id}_{A} \otimes \mu_{A}) \circ (\lr{\hskip 0.4mm,}_{A}^{u} \otimes \mathrm{id}_{A}) + 
(\mu_{A} \otimes \mathrm{id}_{A}) \circ (\mathrm{id}_{A} \otimes \lr{\hskip 0.4mm,}_{A}^{u}) \circ (\tau_{A,A} \otimes \mathrm{id}_{A}).     \]
\end{remark}

\begin{convention}
\label{convention:sign}
Note that the usual definition of double bracket in \cite{vdB}, Section 2.7, is a map of the form 
$\lr{\hskip 0.4mm,}_{A}^{vdB} : A \otimes A \rightarrow A \otimes A$ satisfying some axioms. 
We leave to the reader to verify that the conditions in Definition \ref{definition:po} for our map $\lr{\hskip 0.4mm,}_{A} : A[d] \otimes A[d] \rightarrow A \otimes A$ are equivalent to those given in \cite{vdB}, Section 2.7, for the map $\lr{\hskip 0.4mm,}_{A}^{vdB} : A \otimes A \rightarrow A \otimes A$, where $\lr{a,b}^{vdB}_{A} = \lr{s^{d}a,s^{d}b}_{A}$, and $a, b \in A$.
It is for this reason that, when dealing with specific elements $a, b$ of $A$, it will be convenient to simply write $\lr{a,b}_{A}$ instead of $\lr{a,b}_{A}^{vdB}$ ($= \lr{s^{d}a,s^{d}b}_{A}$). 
\end{convention}

We recall that, given any dg algebra $(A,\mu_{A},\partial_{A})$, $[A,A]$ denotes the dg vector subspace of $A$ generated by $a b - (-1)^{|a| |b|} b a$, for all homogeneous $a, b \in A$. 
Note that we have the isomorphism of dg vector spaces $A[d]/([A,A][d]) \simeq (A/[A,A])[d]$ given by $s^{d}_{A}(a) + ([A,A])[d] \mapsto s^{d}_{A/[A,A]}(a + [A,A])$. 
The following result is proved by the same argument as the one in \cite{vdB}, Corollary 2.4.6.
\begin{proposition}
\label{proposition:lie}
Let $(A,\mu_{A},\partial_{A})$ be a dg algebra provided with a double Poisson bracket $\lr{\hskip 0.4mm,}_{A}$ of degree $-d \in \ZZ$.  
Set $\{ \hskip 0.4mm, \}_{A} : A[d] \otimes A[d] \rightarrow (A/[A,A])[d]$ to be the composition of $\lr{\hskip 0.4mm,}_{A}$, $\mu_{A}$, the canonical projection $A \rightarrow A/[A,A]$ and $s^{d}_{A/[A,A]}$. 
Then, $\{ \hskip 0.4mm, \}_{A}$ induces a map 
\[     (A/[A,A])[d] \otimes (A/[A,A])[d] \rightarrow (A/[A,A])[d]     \] 
of degree zero, which, together with the map 
\[     (A/[A,A])[d] \rightarrow (A/[A,A])[d]     \] 
of degree $1$ induced by $s^{d}_{A} \circ \partial_{A} \circ s^{-d}_{A[d]}$, gives a structure of dg Lie algebra on $(A/[A,A])[d]$. 
\end{proposition}

\section{\texorpdfstring{Cyclic $A_{\infty}$-algebras and pre-Calabi-Yau structures}{Cyclic A-infinity-algebras and pre-Calabi-Yau structures}}
\label{section:cyclic-pre-CY}

\subsection{\texorpdfstring{$A_{\infty}$-algebras}{A-infinity-algebras}}
\label{subsection:Ainfty}

We recall that a \emph{nonunitary $A_{\infty}$-algebra} is a (cohomologically) graded vector space 
$A=\oplus_{n \in \ZZ} A^{n}$ together with a collection of maps $\{ m_{n} \}_{n \in \NN}$, 
where $m_{n} : A^{\otimes n} \rightarrow A$ is a homogeneous morphism of degree $2-n$, satisfying the equation 
\begin{equation}
\tag{$\operatorname{SI}(n)$}
\label{eq:ainftyalgebra}
   \sum_{(r,s,t) \in \mathcal{I}_{n}} (-1)^{r + s t} m_{r + 1 + t} \circ (\mathrm{id}_{A}^{\otimes r} \otimes m_{s} \otimes \mathrm{id}_{A}^{\otimes t}) = 0
\end{equation}
for $n \in \NN$, where $\mathcal{I}_{n} = \{ (r,s,t) \in \NN_{0} \times \NN \times \NN_{0} : r + s + t = n \}$. 
Since we are going to deal exclusively with nonunitary $A_{\infty}$-algebras, from now on, $A_{\infty}$-algebras will always 
be nonunitary, unless otherwise stated. 

\begin{definition}
\label{definition:smallness}
An $A_{\infty}$-algebra $(A,m_{\bullet})$ is said to be
\begin{enumerate}[label={\textup{(\roman*)}}]
\setcounter{enumi}{0} 
\item\label{item:fully-man-ext} \emph{fully manageable} 
if $(A,m_{2},m_{1})$ is a (nonunitary) dg algebra; 
\item\label{item:small} \emph{small} if the multiplications $\{ m_{n} \}_{n \in \NN}$ satisfy that $m_{n} = 0$, for all $n \geq 4$;
\item\label{item:essentially-odd} \emph{essentially odd} if $m_{2i} = 0$, for all $i > 1$. 
\end{enumerate}
\end{definition}

In case \ref{item:fully-man-ext} we also say that $(A,m_{\bullet})$ is a \emph{fully manageable extension} of the dg algebra 
$(A,m_{2},m_{1})$, if we want to emphasize the latter. 

Note that given an essentially odd $A_{\infty}$-algebra, $\operatorname{SI}(2p)$ is equivalent to 
\begin{equation}
\label{eq:ainftyalgebraodd-ev}
   \sum_{r=0}^{2(p-1)} (-1)^{r} m_{2p-1} \circ (\mathrm{id}_{A}^{\otimes r} \otimes m_{2} \otimes \mathrm{id}_{A}^{\otimes (2(p-1)-r)}) - m_{2} \circ (m_{2p-1} \otimes \mathrm{id}_{A} + \mathrm{id}_{A} \otimes m_{2p-1}) = 0,
\end{equation}
for $p \in \NN$, whereas $\operatorname{SI}(2p-1)$ is equivalent to 
\begin{equation}
\label{eq:ainftyalgebraodd-odd}
\begin{split}
   &\delta_{p,2} m_{2} \circ \big(m_{2} \otimes \mathrm{id}_{A} - \mathrm{id}_{A} \otimes m_{2}\big) 
   \\
   &+ \sum_{i=1}^{p} \sum_{r=0}^{2(p-2)} m_{2i-1} \circ \big(\mathrm{id}_{A}^{\otimes r} \otimes m_{2(p-i)+1} \otimes \mathrm{id}_{A}^{\otimes (2(i-2)-r)}\big) = 0,
\end{split}
\end{equation}
for $p \in \NN$.

A \emph{morphism of (nonunitary) $A_{\infty}$-algebras} $f_{\bullet} : A \rightarrow A'$ between two (nonunitary) $A_{\infty}$-algebras $(A,m_{\bullet}^{A})$ and $(A',m_{\bullet}^{A'})$ is a collection of maps $\{ f_{n} \}_{n \in \NN}$, 
where $f_{n} : A^{\otimes n} \rightarrow A'$ is a homogeneous morphism of degree $1-n$ satisfying the equation 
\begin{equation}
\tag{$\operatorname{MI}(n)$}
\label{eq:ainftyalgebramor}
   \sum_{(r,s,t) \in \mathcal{I}_{n}} (-1)^{r + s t} f_{r + 1 + t} \circ (\mathrm{id}_{A}^{\otimes r} \otimes m_{s}^{A} \otimes \mathrm{id}_{A}^{\otimes t})  
   = \sum_{q \in \NN} \sum_{\bar{i} \in \NN^{q, n}} (-1)^{w} m_{q}^{A'} \circ (f_{i_{1}} \otimes \dots \otimes f_{i_{q}}),
\end{equation} 
for $n \in \NN$, where $w = \sum_{j=1}^{q} (j-1) (i_{j} + 1)$ and $\NN^{q, n}$ is the subset of elements $\bar{i}$ of $\NN^{q}$ satisfying that $|\bar{i}| = i_{1} + \dots + i_{q} = n$. 
A morphism is \emph{strict} if $f_{n} = 0$, for all $n \geq 2$. 

\subsection{\texorpdfstring{Cyclic and ultracyclic structures on $A_{\infty}$-algebras}{Cyclic and ultracyclic structures on A-infinity-algebras}}
\label{subsection:ultra-cyclic}

Given $d \in \ZZ$, a \emph{$d$-cyclic (nonunitary) $A_{\infty}$-algebra} is an $A_{\infty}$-algebra $(A,m_{\bullet})$ provided with a nondegenerate bilinear form $\gamma : A \otimes A \rightarrow \kk$ of degree $d$ satisfying that $\gamma \circ \tau_{A,A} = \gamma$ and 
\begin{equation}
\label{eq:ip1}
     \gamma\big(m_{n}(a_{1},\dots,a_{n}),a_{0}\big) = (-1)^{n + |a_{0}|(\sum_{i=1}^{n} |a_{i}|) } \gamma\big(m_{n}(a_{0},\dots,a_{n-1}),a_{n}\big),  
\end{equation}
for all homogeneous $a_{0}, \dots, a_{n} \in A$. 
If we drop the nondegeneracy assumption on $\gamma$ in the previous definition, we will say that $A$ is a \emph{degenerate $d$-cyclic (nonunitary) $A_{\infty}$-algebra}. 

We also introduce the following definition, that will be useful in the sequel. 
In order to do so, given $n \in \NN$ consider the injective map $\xan_{n} : \mathbb{S}_{n} \rightarrow \mathbb{S}_{2n}$, 
sending $\varsigma \in \mathbb{S}_{n}$ to the permutation $\sigma$ defined by $\sigma(2i) = 2 \varsigma(i)$ and 
$\sigma(2i-1) = 2 \varsigma(i)-1$, for all $i \in \{ 1, \dots, n \}$. \footnote{The symbol $\xan$ was downloaded from \url{https://www.charbase.com}.}
A $d$-cyclic (nonunitary) $A_{\infty}$-algebra $(A,m_{\bullet})$ with a nondegenerate bilinear form $\gamma : A \otimes A \rightarrow \kk$ of degree $d$ satisfying that $(A,m_{\bullet})$ is essentially odd is called \emph{$d$-ultracyclic} if, for all $n \in \NN$ and all permutations $\varsigma \in \mathbb{S}_{n}$, we have that 
\begin{equation}
\label{eq:ip1ultra}
\begin{split}
     \gamma&\big(m_{2n-1}(a_{\varsigma(1)},b_{\varsigma(1)},\dots,a_{\varsigma(n-1)},b_{\varsigma(n-1)},a_{\varsigma(n)}),b_{\varsigma(n)}\big) 
     \\
     &= (-1)^{\epsilon(\varsigma^{-1},\bar{a},\bar{b})} \gamma\big(m_{2n-1}(a_{1},b_{1},\dots,a_{n-1},b_{n-1},a_{n}),b_{n}\big),  
\end{split}
\end{equation}
for all homogeneous $a_{1}, b_{1}, \dots, a_{n}, b_{n} \in A$, where $\epsilon(\varsigma^{-1},\bar{a},\bar{b})$ is the sign given 
in \eqref{eq:permeps2} for $\sigma = \xan_{n}(\varsigma^{-1})$ and $\bar{v} = a_{1} \otimes b_{1} \otimes \dots \otimes a_{n} \otimes b_{n}$. 
As before, if we do not assume that $\gamma$ is nondegenerate in the previous definition, we will say that $A$ is a \emph{degenerate $d$-ultracyclic (nonunitary) $A_{\infty}$-algebra}. 

\subsection{Natural bilinear forms and pre-Calabi-Yau structures}
\label{subsection:nat-bilinear}

Moreover, as it will be useful later, given a cyclic $A_{\infty}$-algebra $(A,m_{\bullet})$ with a nondegenerate bilinear form $\gamma$ and $n \in \NN$, we will define the linear map $\operatorname{SI}(n)_{\gamma} : A^{\otimes (n+1)} \rightarrow \kk$ 
by
\begin{equation}
\tag{$\operatorname{SI}(n)_{\gamma}$}
\label{eq:ainftyalgebragamma}
   \sum_{(r,s,t) \in \mathcal{I}_{n}} (-1)^{r + s t} \gamma \circ \big(m_{r + 1 + t} \circ (\mathrm{id}_{A}^{\otimes r} \otimes m_{s} \otimes \mathrm{id}_{A}^{\otimes t}) \otimes \mathrm{id}_{A}\big).
\end{equation}
Note that the $(A,m_{\bullet})$ being a cyclic $A_{\infty}$-algebra is equivalent to the vanishing of $\operatorname{SI}(n)_{\gamma}$, for all $n \in \NN$. 

For the following definition, we first recall the definition of the \emph{natural bilinear form of degree $d \in \ZZ$} associated with any (cohomologically) graded vector space $A = \oplus_{n \in \ZZ} A^{n}$. 
First, set $\partial_{d}A = A \oplus A^{\#}[d]$. 
For clarity, we will denote the suspension map $s^{d}_{A^{\#}} : A^{\#} \rightarrow A^{\#}[d]$ simply by $t$, 
and any element of $A^{\#}[d]$ will be thus denoted by $tf$, for $f \in A^{\#}$.
Define now the bilinear form 
\[
 \gan_{A} : \partial_{d}A \otimes \partial_{d}A \rightarrow \kk
\]
 by 
\begin{equation}
\label{eq:natform} 
     \gan_{A}(tf,a) = (-1)^{|a| |tf|} \gan_{A}(a,tf) = f(a),
     \text{ $\phantom{x}$ and $\phantom{x}$ }
     \gan_{A}(a,b)  = \gan_{A}(tf,tg) = 0,
\end{equation}
for all homogeneous $a,b \in A$ and $f, g \in A^{\#}$. 
Note that $\gan_{A}$ has degree $d$. 
If there is no risk of confusion, we shall denote $\gan_{A}$ simply by $\gan$. \footnote{The symbol $\gan$ was downloaded from \url{https://www.charbase.com}.}

We recall the following crucial definition from \cite{KV18}.
\begin{definition}
Given $d \in \ZZ$, a \emph{$d$-pre-Calabi-Yau (algebra) structure} on a (cohomologically) graded vector space $A = \oplus_{n \in \ZZ} A^{n}$ is the datum of a $(d-1)$-cyclic $A_{\infty}$-algebra on the graded vector space 
$\partial_{d-1}A = A \oplus A^{\#}[d-1]$ 
for the natural bilinear form $\gan_{A} : \partial_{d-1}A \otimes \partial_{d-1}A \rightarrow \kk$ of degree $d-1$ defined in \eqref{eq:natform} such that the corresponding multiplications $\{ m_{n} \}_{n \in \NN}$ of $\partial_{d-1}A$ 
satisfy that $m_{n}(A^{\otimes n}) \subseteq A$, for all $n \in \NN$. 
A $0$-pre-Calabi-Yau algebra will be simply called a \emph{pre-Calabi-Yau algebra}. 
\label{definition-pre-d-CY}
\end{definition}

This implies in particular that the maps $\{ m_{n}|_{A^{\otimes n}} \}_{n \in \NN}$ 
define an $A_{\infty}$-algebra structure on $A$ such that its canonical inclusion into $\partial_{d-1}A$ 
is a strict morphism of $A_{\infty}$-algebras. 

\subsection{\texorpdfstring{Good and nice $A_{\infty}$-algebras}{Good and nice A-infinity-algebras}}
\label{subsection:good-algebras}

We will now introduce the following terminology that will be useful in the sequel.
Let us first fix some notation. 
Assume that there is a decomposition $B_{0} \oplus B_{1}$ of a graded vector space $B$. 
In many of our examples, $B_{0}$ will be a graded vector space $A$ and $B_{1}$ will be $A^{\#}[d-1]$. 
Then, for any odd integer $n \in \NN$, the decomposition $B = B_{0} \oplus B_{1}$ induces a canonical decomposition
\[     B^{\otimes n} = T_{n,g} \oplus T_{n,b},     \]
where 
\begin{equation}
   T_{n,g} = \bigoplus_{\bar{i} \in \mathscr{I}_{n}} B_{i_{1}} \otimes \dots \otimes B_{i_{n}}, \hskip 5mm 
 T_{n,b} = \bigoplus_{\bar{i} \in \{0,1\}^{n} \setminus \mathscr{I}_{n}} B_{i_{1}} \otimes \dots \otimes B_{i_{n}},  
\label{etiqueta-T-n-g}
\end{equation}  
and
\[     \mathscr{I}_{n} = \big\{ \bar{i} = (i_{1},\dots,i_{n})\in \{0,1\}^{n} : \text{ $i_{j} \neq i_{j+1}$ for all $j \in \{ 1, \dots, n - 1 \}$ } \big\}.     \]
Note that $T_{1,b} = 0$. 
A map $m_{n} : B^{\otimes n} \rightarrow B$ will be called \emph{good} if $m_{n}|_{T_{n,b}}$ vanishes and $m_{n}(B_{i_{1}} \otimes \dots \otimes B_{i_{n}}) \subseteq B_{i_{1}}$, 
for all $(i_{1},\dots,i_{n}) \in \mathscr{I}_{n}$. 

\begin{definition}
\label{definition:terminology2}
Let $B$ be an $A_{\infty}$-algebra provided with an extra decomposition $B = B_{0} \oplus B_{1}$. 
We say that $B$ is 
\begin{enumerate}[label={\textup{(\roman*)}}]
\setcounter{enumi}{0} 
\item\label{item:good-algebra}
\emph{good} if the $A_{\infty}$-algebra structure is essentially odd and for every odd integer $n \in \NN$ the multiplication map $m_{n}$ is good; 
\setcounter{enumi}{1} 
\item\label{item:special-algebra}
\emph{special} if the $A_{\infty}$-algebra structure is essentially odd and $(d-1)$-ultracyclic; 
\setcounter{enumi}{2} 
\item\label{item:nice-algebra}
\emph{nice} if it is good and small (see Definition \ref{definition:smallness}, \ref{item:small}).
\end{enumerate}
\end{definition}

All these definitions apply in particular to a $d$-pre-Calabi-Yau structure on $A$, where we take $B = \partial_{d-1}A = A \oplus A^{\#}[d-1]$. 

\section{Nice pre-Calabi-Yau structures and double Poisson dg algebras}
\label{section:core}

\subsection{Relation between objects}
\label{subsection:main1}

We first recall that, given a (nonunitary) dg algebra $A$ with product $\mu_{A}$ and differential $\partial_{A}$, then 
$A^{\#}$ is naturally a dg bimodule over $A$ via 
\[     (a \cdot f \cdot b) (c) = (-1)^{|a|(|f|+|b|+|c|)} f(b c a),     \]
for all homogeneous $a,b,c \in A$ and $f \in A^{\#}$. 
Moreover, if $(M,\partial_{M})$ is any dg bimodule over $A$ and $d \in \ZZ$, then the dg vector space 
$M[d]$ is a dg bimodule over $A$ via $a \cdot s^{d}_{M}(m) \cdot b = (-1)^{d|a|} s^{d}_{M}(a \cdot m \cdot b)$, for all homogeneous $a, b \in A$ and $m \in M$. 
In particular, $A^{\#}[d-1]$ is a dg bimodule over $A$. 
For simplicity, we will write the product of $A$ and its action on any dg bimodule $M$ by juxtaposition, or a small dot.  

Moreover, given a dg bimodule $M$ over a (nonunitary) dg algebra $A$, consider the dg vector space $A \oplus M$ with the product $(a,m) \cdot (a',m') = (a a' , m \cdot a' + a \cdot m')$. 
It is easy to verify that the dg vector space $A \oplus M$ provided with the previous product is a (nonunitary) dg algebra. 
In particular, we see that $A \oplus A^{\#}[d-1]$ is a (nonunitary) dg algebra. 
We leave to the reader to verify the easy assertion that this dg algebra together with the natural bilinear form of degree $d-1$ defined in \eqref{eq:natform} is in fact a $d$-pre-Calabi-Yau structure, by taking $m_{1}$ to be the differential of $A \oplus A^{\#}[d-1]$, $m_{2}$ its product, and $m_{n} = 0$, for all $n \geq 3$. 

\begin{definition}
\label{definition:terminology23}
Let $(A, \mu_{A}, \partial_{A})$ be a (nonunitary) dg algebra, and consider the $d$-pre-Calabi-Yau structure on $A$ defined 
by the dg algebra structure of $A \oplus A^{\#}[d-1]$ described before, 
together with the natural bilinear form of degree $d-1$ defined in \eqref{eq:natform}. 
A $d$-pre-Calabi-Yau structure $\{ m_{n} \}_{n \in \ZZ}$ on $A$ is called \emph{manageable} if $m_{2}$ coincides with the product of $A \oplus A^{\#}[d-1]$ considered before, and \emph{fully manageable} if we also have that $m_{1}$ is the differential of $A \oplus A^{\#}[d-1]$.
\end{definition}

The following result generalizes \cite{IK17}, Theorem 4.2. 
\begin{theorem}
\label{theorem:main1}
Let $d \in \ZZ$, and let $A = \oplus_{n \in \ZZ} A^{n}$ be a (nonunitary) dg algebra with product $\mu_{A}$ and differential $\partial_{A}$. 
Consider the dg algebra algebra structure on $A \oplus A^{\#}[d-1]$ explained above, with 
product $m_{2}$ and differential $m_{1}$, as well as the natural bilinear form on it of degree $d-1$ defined in \eqref{eq:natform}.
Given any nice and fully manageable $d$-pre-Calabi-Yau structure $\{m_{\bullet}\}_{\bullet \in \NN}$ on $A$, 
define the map $\lr{\hskip 0.4mm , } : A \otimes A \rightarrow A \otimes A$ by 
\begin{equation}
\label{eq:m3cor}
     (f \otimes g)\big(\lr{a,b}\big) = s_{f,g}^{a,b}\, \gan\big(m_{3}(b,tg,a),tf\big),
\end{equation}
for all homogeneous $a,b \in A$ and $f,g \in A^{\#}$, where $s_{f,g}^{a,b} =  (-1)^{|b|(|a|+|g|+1)}$.  
Then, $\lr{\hskip 0.4mm , }$ is a double Poisson bracket of degree $-d$ on the dg algebra $A$.
Moreover, the map
\begin{equation}
\label{eq:bij}
 \bigg\{ \begin{matrix} 
 \text{fully manageable nice $d$-pre-CY}
 \\
 \text{structures $\{m_{\bullet}\}_{\bullet \in \NN}$ on $A$}
 \end{matrix} 
\bigg\} 
\longrightarrow
 \bigg\{ \begin{matrix}
 \text{double Poisson brackets on}
 \\
 \text{$A$ of degree $-d$}
 \end{matrix} 
\bigg\} 
\end{equation}
given by sending $m_{3}$ to the double Poisson bracket determined by \eqref{eq:m3cor} is a bijection. 
\end{theorem}
\begin{proof}
We are only going to consider $s_{f,g}^{a,b}$ when \eqref{eq:m3cor} is not trivially zero, \textit{i.e.} 
if $|f| + |g| = |a| + |b| + d$. 
Note that this last identity implies also that $s_{f,g}^{a,b} = (-1)^{|b|(|f|+d)}$. 

We will first prove that $\lr{\hskip 0.4mm , }$, as defined in \eqref{eq:m3cor}, is a double Poisson bracket on the dg algebra $A$. 
We remark that we will be using Convention \ref{convention:sign}. 
Let us start with the antisymmetric property \ref{item:dpa1} in Definition \ref{definition:po}, \textit{i.e.} 
\begin{equation}
\label{eq:anti}
     \tau_{A,A} \big(\lr{b,a}\big) = - (-1)^{(|a|-d)(|b|-d)} \lr{a,b},
\end{equation}     
for all homogeneous $a, b \in A$.
Evaluating $g \otimes f$ at both sides of the previous equation, where $f ,g \in A^{\#}$ are homogeneous, 
it is clear that \eqref{eq:anti} is equivalent to 
\begin{equation}
\label{eq:anti2}
     (g \otimes f) \big(\lr{a,b}\big) = - (-1)^{(|a| |b| + |f| |g| + d(|a|+|b|+1))} (f \otimes g)\big(\lr{b,a}\big),
\end{equation} 
for all homogeneous $a, b \in A$ and $f, g \in A^{\#}$. 
Using \eqref{eq:m3cor} on each side, we obtain that \eqref{eq:anti2} is equivalent to 
\begin{equation}
\label{eq:anti3}
\begin{split}
     &s_{g,f}^{a,b} \gan\big(m_{3}(b,tf,a),tg\big) 
     \\
     &= - (-1)^{(|a| |b| + |f| |g| + d(|a|+|b|+1))} s_{f,g}^{b,a} \gan\big(m_{3}(a,tg,b),tf\big).
\end{split}
\end{equation} 
On the left-hand side, using the cyclicity property of $\gan$, we obtain that 
\begin{equation}
\label{eq:anti4}
\begin{split}
     &s_{g,f}^{a,b} \gan\big(m_{3}(b,tf,a),tg\big) 
     \\
     &= s_{g,f}^{a,b} (-1)^{|tg|(|b|+|tf|+|a|)+3} \,\gan\big(m_{3}(tg,b,tf),a\big) 
     \\
     &= s_{g,f}^{a,b} (-1)^{|tg|(|b|+|tf|+|a|)+|a|(|f|+|b|+|g|)}\, \gan\big(m_{3}(a,tg,b),tf\big).  
\end{split}
\end{equation}
Hence, comparing \eqref{eq:anti3} and \eqref{eq:anti4}, we see that \eqref{eq:anti} holds if and only if 
\begin{equation}
\label{eq:anti5}
\begin{split}
     s_{g,f}^{a,b} = (-1)^{|a| |f| + |g| |b| + d(|a| + |b|)} s_{f,g}^{b,a},
\end{split}
\end{equation}
where we have used that $|a| + |b| + |f| + |g| + d = 0$ ($\operatorname{mod} 2$). 
Replacing $s_{g,f}^{a,b}$ by its definition, we see that \eqref{eq:anti5} holds, so \eqref{eq:anti} does it as well, as was to be shown. 

Let us now prove the Leibniz property \ref{item:dpa2} in Definition \ref{definition:po}, \textit{i.e.} 
\begin{equation}
\label{eq:leib}
     \lr{c,ab} = \lr{c,a} b + (-1)^{(|c|-d)|a|} a\lr{c,b},
\end{equation}     
for all homogeneous $a, b, c \in A$.
In order to do so, consider the identity \eqref{eq:ainftyalgebra} for $n=4$ for the $A_{\infty}$-algebra structure of $A \oplus A^{\#}[d-1]$ 
evaluated at $a \otimes b \otimes tf \otimes c$, where $a, b, c \in A$ and $f \in A^{\#}$ are homogeneous elements. 
This gives 
\begin{equation}
\label{eq:leib2}
     m_{3}(a b, t f, c) - m_{3}(a, b. tf, c) - (-1)^{|a|} a. m_{3}(b,tf,c) = 0,
\end{equation}     
for all homogeneous $a, b, c \in A$ and $f \in A^{\#}$. 
By applying $\gan(\place,tg)$, for a general homogeneous $g \in A^{\#}$, we see that \eqref{eq:leib2} is tantamount to 
\begin{equation}
\label{eq:leib3}
     \gan\big(m_{3}(a b, t f, c),tg\big) = \gan\big(m_{3}(a, b. tf, c),tg\big) + (-1)^{|a|} \gan\big(a. m_{3}(b,tf,c),tg\big) = 0.
\end{equation}     
Using definition \eqref{eq:m3cor}, we see that the first term of \eqref{eq:leib3} is precisely
\begin{equation}
\label{eq:leib31}
     s_{g,f}^{c,ab} (g \otimes f)\big(\lr{c,ab}\big).
\end{equation} 
Similarly, using the identity $b.(tf) = (-1)^{|b|(d-1)} t(b.f)$, for all homogeneous $b \in A$ and $f \in A^{\#}$, and \eqref{eq:m3cor}, 
the second term of \eqref{eq:leib3} becomes 
\begin{equation}
\label{eq:leib32}
     s_{g,b.f}^{c,a} (-1)^{|b|(d-1)} \big(g \otimes (b.f)\big)\big(\lr{c,a}\big).
\end{equation} 
Using the identity $(g \otimes (b.f))(v \otimes w) = (-1)^{|b|(|f|+|v|+|w|)} (g \otimes f)(v \otimes (w.b))$, for all homogeneous $b, v , w \in A$ and $f, g \in A^{\#}$, and the fact that $|\lr{c,a}| = |c| + |a| - d$, we conclude that \eqref{eq:leib32} is equal to 
\begin{equation}
\label{eq:leib32bis}
     s_{g,b.f}^{c,a} (-1)^{|b|(|f|+|a|+|c|+1)} (g \otimes f)\big(\lr{c,a}b\big).
\end{equation}
Finally, using the cyclicity of $\gan$ we see that the third term of \eqref{eq:leib3} is 
\begin{equation}
\label{eq:leib33}
\begin{split}
     &(-1)^{|a|+|tg|(|a|+|b|+|tf|+|c|-1)} \,\gan\big(tg.a,m_{3}(b,tf,c)\big) 
     \\
     &= (-1)^{|a|+|tg|(|a|+|b|+|tf|+|c|-1)+(|a|+|tg|)(|tf|+|b|+|c|-1)}\, \gan\big(m_{3}(b,tf,c),tg.a\big)
     \\
     &= (-1)^{|a|(|c|+|b|+|g|+|f|)} s_{g.a,f}^{c,b} \big((g.a) \otimes f\big)\big(\lr{c,b}\big),
\end{split}
\end{equation}
where we used the super symmetry of $\gan$ in the second identity, and $tg.a = t(g.a)$ in the last equality.
Using the identity $((g.a) \otimes f)(v \otimes w) = (-1)^{|a||f|} (g \otimes f)(a.v \otimes w)$, for all homogeneous $a, v , w \in A$ and $f, g \in A^{\#}$, we conclude that \eqref{eq:leib33} is equal to 
\begin{equation}
\label{eq:leib33bis}
     (-1)^{|a|(|c|+|b|+|g|)} s_{g.a,f}^{c,b} (g \otimes f)\big(a\lr{c,b}\big).
\end{equation}
Then, multiplying \eqref{eq:leib3} by $s_{g,f}^{c,ab}$ and replacing the corresponding terms by the ones given by 
\eqref{eq:leib31}, \eqref{eq:leib32bis} and \eqref{eq:leib33bis}, we get 
\begin{equation}
\label{eq:leib4}
\begin{split}
     (g \otimes f)\big(\lr{c,ab}\big) &= s_{g,f}^{c,ab} s_{g,b.f}^{c,a} (-1)^{|b|(|f|+|c|+|a|+1)} (g \otimes f)\big(\lr{c,a} b\big) 
     \\
     &+ s_{g,f}^{c,ab} s_{g.a,f}^{c,b} (-1)^{|a|(|g|+|c|+|b|)} (g \otimes f)\big(a \lr{c,b}\big).
\end{split}
\end{equation} 
Hence, \eqref{eq:leib} holds if and only if 
\begin{equation}
\label{eq:leib5}
     s_{g,f}^{c,ab}  = s_{g,b.f}^{c,a} (-1)^{|b|(|f|+|c|+|a|+1)}, \text{ and } (-1)^{|a|(|c|-d)}=s^{c,ab}_{g,f}s^{c,b}_{g.a,f}(-1)^{|a|(|g|+|c|+|b|)},
\end{equation} 
Using the definition of $s_{g,f}^{b,a}$ together with $|f|+|g| = |a|+|b|+|c|-d$ and $|a| = |a|^{2}$ ($\operatorname{mod} 2$), 
one can easily verify \eqref{eq:leib5}, so \eqref{eq:leib} holds as was to be shown. 
This proves the Leibniz property \ref{item:dpa2} in Definition \ref{definition:po}.

\begin{remark}
\label{remark:signos}
Assuming that $s_{g,f}^{a,b}$ is just a function of the degrees $|a|$, $|b|$, $|f|$ and $|g|$ (satisfying that $|a| + |b| + |f| + |g| + d =0$ ($\operatorname{mod} 2$)), one can in fact show that our choice for $s_{g,f}^{a,b}$ is the unique solution of 
\eqref{eq:anti5} and \eqref{eq:leib5}, up to multiplicative constant $\pm 1$. 
This is in fact how we found such an expression. 
\end{remark}

Let us now show that $\lr{\hskip 0.4mm , }$ is a closed morphism of dg vector spaces, \textit{i.e.} 
\begin{equation}
\label{eq:closed}
     \big(\partial_{A} \otimes \mathrm{id}_{A} + \mathrm{id}_{A} \otimes \partial_{A}\big) \big(\lr{a,b}\big) = \lr{\partial_{A}(a),b} + (-1)^{|a|+d} \lr{a,\partial_{A}(b)},
\end{equation}     
for all homogeneous $a, b \in A$.  
In order to prove this, consider the identity \eqref{eq:ainftyalgebra} for $n=3$ for the $A_{\infty}$-algebra structure of $A \oplus A^{\#}[d-1]$ 
evaluated at $b \otimes tg \otimes a$, where $a, b \in A$ and $g \in A^{\#}$ are homogeneous elements, which gives
\begin{equation}
\label{eq:closed2}
\begin{split}
     m_{1}\big(m_{3}(b, t g, a)\big) &+ m_{3}\big(m_{1}(b), tg, a\big) + (-1)^{|b|} m_{3}\big(b,m_{1}(tg),a\big) 
     \\ &+ (-1)^{|b|+|g|+d-1} m_{3}\big(b,tg,m_{1}(a)\big) = 0.
\end{split}
\end{equation}    
Applying $\gan(\place,tf)$, for an arbitrary homogeneous $f \in A^{\#}$, we see that \eqref{eq:closed2} is tantamount to 
\begin{equation}
\label{eq:closed3}
\begin{split}
     &\gan\Big(m_{1}\big(m_{3}(b, t g, a)\big),tf\Big) + \gan\Big(m_{3}\big(m_{1}(b), tg, a\big),tf\Big) 
     \\ 
     &+ (-1)^{|b|} \gan\Big(m_{3}\big(b,m_{1}(tg),a\big),tf\Big) + (-1)^{|b|+|g|+d-1}\gan\Big(m_{3}\big(b,tg,m_{1}(a)\big),tf\Big) = 0.
\end{split}
\end{equation}  
Applying the cyclicity and super symmetry properties of $\gan$ in the first term, as well as the fact that $m_{1}(th) = (-1)^{|h|+d} t(h \circ \partial_{A})$ for any homogeneous element $h$ in $A^{\#}$ in the first and third terms, we see that \eqref{eq:closed3} is equivalent to 
\begin{equation}
\label{eq:closed4}
\begin{split}
     &-(-1)^{|b|+|g|+|a|+|f|}\, \gan\big(m_{3}(b, t g, a),t(f\circ \partial_{A})\big)  
     \\
     &+ \gan\Big(m_{3}\big(m_{1}(b), tg, a\big),tf\Big) + (-1)^{|b|+|g|+d} \,\gan\Big(m_{3}\big(b,t(g\circ \partial_{A}),a\big),tf\Big) 
     \\
     &\phantom{ + \gan(m_{3}(m_{1}(b), tg, a),tf) } + (-1)^{|b|+|g|+d-1}\,\gan\Big(m_{3}\big(b,tg,m_{1}(a)\big),tf\Big) = 0.
\end{split}
\end{equation} 
By \eqref{eq:m3cor} and multiplying \eqref{eq:closed4} by $(-1)^{|b|(|a|+|g|)+|g|}$, we see that \eqref{eq:closed4} is tantamount to 
\begin{equation}
\label{eq:closed5}
\begin{split}
     &-(-1)^{|a|+|f|} \big((f\circ \partial_{A}) \otimes g\big)\big(\lr{a,b}\big) - (-1)^{|a|+|b|} (f \otimes g)\big(\lr{a,\partial_{A}(b)}\big)     
     \\
     &+ (-1)^{|b|+d} \big(f \otimes (g \circ \partial_{A})\big)\big(\lr{a,b}\big) - (-1)^{|b|+d} (f \otimes g)\big(\lr{\partial_{A}(a),b}\big) = 0,
\end{split}
\end{equation} 
where we have used that $m_{1}|_{A} = \partial_{A}$, and $|a|+|b|+|f|+|g|+d+1=0$ ($\operatorname{mod} 2$).
Now, using the Koszul sign rule, we obtain that \eqref{eq:closed5} is precisely 
\begin{equation}
\label{eq:closed6}
     (f \otimes g)(\partial_{A} \otimes \mathrm{id}_{A} + \mathrm{id}_{A} \otimes \partial_{A}) \big(\lr{a,b}\big) = (f \otimes g)\big(\lr{\partial_{A}(a),b} + 
      (-1)^{|a|+d} \lr{a,\partial_{A}(b)}\big),
\end{equation} 
for all for all homogeneous $a, b \in A$ and $f, g \in A^{\#}$, which is clearly equivalent to \eqref{eq:closed}. 

We shall now prove the double Jacobi identity (see \ref{item:dpa3} in Definition \ref{definition:po}), which can be explicitly written 
as
\begin{equation}
 \label{eq:Jacobi-goal}
 \begin{split}
 \lr{c,\lr{b,a}}_L &+(-1)^{(|c|+d) (|a|+|b|)}\sigma \lr{b,\lr{a,c}}_L
 \\
 &+(-1)^{(|a|+d)(|b|+ |c|)}\sigma^{2} \lr{a,\lr{c,b}}_L=0,
\end{split}
\end{equation}
for arbitrary homogeneous elements $a,b,c \in A$, where $\sigma \in \mathbb{S}_{3}$ is the unique cyclic permutation sending $1$ to $2$.
In order to do so, consider \eqref{eq:ainftyalgebra} for $n=5$ evaluated at $a \otimes tf \otimes b \otimes tg \otimes c$, where $a,b,c\in A$ and $f,g\in A^{\#}$ are homogeneous elements. 
It gives
\begin{equation}
\label{eq:MC-5}
\begin{split}
 m_3\big(m_3(a,tf,b),tg,c\big)&+(-1)^{| a|}m_3\big(a,m_3(tf,b,tg),c\big)
 \\
 &+(-1)^{|a|+| tf|}m_3\big(a,tf,m_3(b,tg,c)\big) = 0.
\end{split}
\end{equation}
It is equivalent to the following identity, when we apply $\gan(\place,th)$ for an arbitrary homogeneous element $h\in A^{\#}$,
\begin{equation}
 \label{eq:MC5-pairing}
\begin{split}
  \gan\Big(m_3\big(m_3(a,tf,b),tg,c\big),th\Big) &+(-1)^{|a|}\,\gan\Big(m_3\big(a,m_3(tf,b,tg),c\big),th\Big)
  \\
  &+(-1)^{| a|+| tf|}\,\gan\Big(m_3\big(a,tf,m_3(b,tg,c)\big),th\Big) = 0.
\end{split}
\end{equation}

The following result will be essential to prove the double Jacobi identity.
\begin{fact}
\label{lemma-Jacobi}
Let $a, b, c \in A$ and $f, g, h \in A^{\#}$ be homogeneous elements. 
Then, 
 \begin{equation}
  (f\otimes g\otimes h)\big(\lr{a,\lr{b,c}}_L\big)=\Box^{a,b,c}_{f,g,h}\,\gan\Big(m_3\big(m_3(c,th,b),tg,a\big),tf\Big),
\label{eq:relacion-lemma-tecnico}
\end{equation}
where
\begin{equation}
 \label{eq:symbol-lemma}
  \Box^{a,b,c}_{f,g,h}= (-1)^{|h|(|f|+|g|) + d (|f|+|h|) + (|b|+|c|+|h|)(|a|+|b|+|g|)} 
\end{equation}
\end{fact}
\begin{proof}
First note that, since $(f\otimes g)$ is a functional applied to $\lr{a,\lr{b,c}'}$, we can assume without loss of generality that 
$| \lr{a,\lr{b,c}'} |=| f|+| g|$. 
As a consequence, 
\begin{equation}
\label{eq:rel-lem-tec1}
 (f\otimes g\otimes h)\big(\lr{a,\lr{b,c}}_{L}\big)=(-1)^{| h| (|f|+|g|)}(f\otimes g)\big(\lr{a,\lr{b,c}'}\big)h\big(\lr{b,c}''\big)
\end{equation}         
Using \eqref{eq:m3cor} we see that the right member of \eqref{eq:rel-lem-tec1} is given by
\begin{equation*}
\label{eq:rel-lem-tec2}
\begin{split}
&(-1)^{| h(| f|+| g)}s_{f,g}^{a,\lr{b,c}'}\gan\Big(m_3\big(\lr{b,c}^{\prime},tg,a\big),tf\Big) h\big(\lr{b,c}''\big)
\\
&=(-1)^{| h|(| f|+| g|) + (|b| + |c| - |h| +d)(|a|+|g|+1)}\gan\bigg(m_3\Big(\lr{b,c}' h\big(\lr{b,c}''\big),tg,a\Big),tf\bigg),
\end{split}
\end{equation*}
where we have used that $|\lr{b,c}'| = |\lr{b,c}| - |\lr{b,c}''| = |b| + |c| + d - |h|$. 
Hence, the previous equalities together with $| a|+| b|+| c|+| f|+| g|+| h|=0$ ($\operatorname{mod} 2$) tell us that the identity \eqref{eq:relacion-lemma-tecnico} is tantamount to 
\begin{equation}
\label{eq:rel-lem-tec3}
\lr{b,c}^{\prime}h\big(\lr{b,c}^{\prime\prime}\big) = (-1)^{d(|b|+|c|+1)+ (|b|+|c|+|h|)(|b|+1)} m_3(c,th,b),
\end{equation}
which is equivalent to 
\begin{equation}
\label{eq:rel-lem-tec3bis}
l\big(\lr{b,c}'\big)h\big(\lr{b,c}''\big) = (-1)^{d(|b|+|c|+1)+ (|b|+|c|+|h|)(|b|+1)} l\big(m_3(c,th,b)\big),
\end{equation}
for all $l \in A^{\#}$ homogeneous of degree $|c| + |b| + |h| - d$.
The left member of \eqref{eq:rel-lem-tec3bis} is given by 
\begin{equation}
\label{eq:rel-lem-tec4l}
     l\big(\lr{b,c}'\big)h\big(\lr{b,c}''\big)=(-1)^{(|c| + |b| + |h| - d) | h |} (l \otimes h)\big(\lr{b,c}\big),     
\end{equation}
whereas, on the right member,
\begin{equation}
\label{eq:rel-lem-tec4r}
     l\big(m_3(c,th,b)\big) = (-1)^{(|c| + |b| + |h| - d)(|c| + |b| + |h|+1)} \gan\big(m_{3}(c,th,b),tl\big),     
\end{equation}
by the super symmetry of $\gan$. 
By \eqref{eq:m3cor} and \eqref{eq:rel-lem-tec4r}, the right member of \eqref{eq:rel-lem-tec3bis} gives 
\[     (-1)^{(|c| + |b| + |h| + d) | h |} (l \otimes h)\big(\lr{b,c}\big),          \]
which coincides with \eqref{eq:rel-lem-tec4l}, proving \eqref{eq:rel-lem-tec3bis}, as was to be shown. 
\end{proof}

By Fact \ref{lemma-Jacobi}, the first term of the left member of \eqref{eq:MC5-pairing} is precisely 
\[     \gan\Big(m_3\big(m_3(a,tf,b),tg,c\big),th\Big)=\Box^{c,b,a}_{h,g,f}( h\otimes g\otimes f)\big(\lr{c,\lr{b,a}}_L\big),     \]
or, more explicitly, 
\begin{equation}
 \label{eq:first-Jacobi}
 (-1)^{d(| f|+| h|) + |f|(|g|+|h|) + (|b|+|c|+|g|)(|a|+|b|+|f|)} (h\otimes g\otimes f)\big(\lr{c,\lr{b,a}}_L\big).
\end{equation}

Next, using the cyclicity of $\gan$ twice and the fact that $| a|+| b|+| c|+| f|+| g|+| h|=0$ ($\operatorname{mod} 2$), we see that the second term of the left member of \eqref{eq:MC5-pairing} is
\begin{equation}
\label{eq:second-Jacobi}
     (-1)^{d |a| + (|b|+| f|+| g|-1)(| a|+|c|+|h|+d-1)}\,\gan\big(m_3(c,th,a),m_3(tf,b,tg)\big).     
\end{equation}
By the supersymmetry of $\gan$, \eqref{eq:second-Jacobi} coincides with 
\begin{equation*}
\begin{split}
&(-1)^{d |a| + |b|+|f|+| g|-1}\,\gan\big(m_3(tf,b,tg),m_3(c,th,a)\big)
\\
&=(-1)^{d |a|+(|b|+|f|+| g|)(|a|+|c|+|h|-d+1)}\,\gan\Big(m_3\big(m_3(c,th,a),tf,b\big),tg\Big),
\end{split}
\end{equation*}
where we have used the cyclicity in the second line. 
Fact \ref{lemma-Jacobi} tells us finally that the second term of the left member of \eqref{eq:MC5-pairing} is 
\begin{equation*}
\begin{split}
&(-1)^{| a |}\gan\Big(m_3\big(a,m_3(tf,b,tg),c\big),th\Big)
\\
&=(-1)^{d |a|+(|b|+|f|+| g|)(|a|+|c|+|h|-d+1)}\Box^{b,a,c}_{g,f,h}(g\otimes f\otimes h)\big(\lr{b,\lr{a,c}}_L\big),
\end{split}
\end{equation*}
which, by \eqref{eq:unipermvecfun}, is equal to 
\begin{equation}
\label{eq:second-Jacobibis}
(-1)^{|h|(|f|+|g|)+d |a|+(|b|+|f|+| g|)(|a|+|c|+|h|-d+1)}\Box^{b,a,c}_{g,f,h} (h\otimes g\otimes f)\big(\sigma \lr{b,\lr{a,c}}_L\big).
\end{equation}
Using the definition of $\Box^{b,a,c}_{g,f,h}$ as well as $|a|+|b|+|c|+f|+|g|+|h|=0$ ($\operatorname{mod} 2$) and $|x|^{2} = |x|$($\operatorname{mod} 2$), one obtains that \eqref{eq:second-Jacobibis} is given by 
\begin{equation}
\label{eq:second-Jacobibis2}
(-1)^{d (|c| + |g|)+(|b|+|f|+|g|)(|a|+|b|+|f|)} ( h\otimes g\otimes f)\big(\sigma \lr{b,\lr{a,c}}_L\big).
\end{equation}

Finally, using the cyclicity of $\gan$ twice and $| a|+| b|+| c|+| f|+| g|+| h|=0$ ($\operatorname{mod} 2$), we see that the third term of the right member of \eqref{eq:MC5-pairing} is given by 
\[
(-1)^{d(|a| + |f|)}\gan\Big(m_3\big(m_3(b,tg,c),th,a\big),tf\Big),
\]
which, by Fact \ref{lemma-Jacobi}, coincides with
\begin{equation}
\label{eq:third-Jacobi}
(-1)^{d(| a|+| f|)}\Box^{a,c,b}_{f,h,g} (f\otimes h\otimes g)\big(\lr{a,\lr{c,b}}_L\big).
\end{equation}
By \eqref{eq:unipermvecfun}, we see that \eqref{eq:third-Jacobi} coincides with
\begin{equation*}
(-1)^{d(| a|+| f|) + |f|(|g|+|h|)}\Box^{a,c,b}_{f,h,g} (h\otimes g\otimes f)\big(\sigma^{2} \lr{b,\lr{a,c}}_L\big),
\end{equation*}
which can be further reduced to the form 
\begin{equation}
\label{eq:third-Jacobibisbis}
(-1)^{d(| a|+| g|) + |h|(|f|+|g|) + (|b|+|c|+|g|)(|a|+|c|+|h|)} (h\otimes g\otimes f)\big(\sigma^{2} \lr{b,\lr{a,c}}_L\big).
\end{equation}

Since $| a|+| b|+| c|+| f|+| g|+| h|=0$ ($\operatorname{mod} 2$), it is fairly easy to prove that the product of the sign appearing in \eqref{eq:first-Jacobi} and the one in \eqref{eq:second-Jacobibis2} 
is precisely $(-1)^{(|c|+d)(|a|+|b|)}$, whereas the product of the sign appearing in \eqref{eq:first-Jacobi} and the one in \eqref{eq:third-Jacobibisbis} is $(-1)^{(|a|+d)(|b|+|c|)}$. 
Using these results, plugging \eqref{eq:first-Jacobi}, \eqref{eq:second-Jacobibis2} and \eqref{eq:third-Jacobibisbis} in \eqref{eq:MC5-pairing}, and multiplying the latter by the sign appearing in \eqref{eq:first-Jacobi}, we obtain precisely \eqref{eq:Jacobi-goal}, as was to be shown. 

To sum up, we have proved that, via \eqref{eq:m3cor}, $(A,\mu_{A},\partial_{A})$ is endowed with a double Poisson dg structure, which in turn means that the map \eqref{eq:bij} is well-defined. 
Furthermore, notice that, if $\{m_{\bullet}\}_{\bullet \in \NN}$ is a small and fully manageable $d$-pre-Calabi-Yau structure on $A$ and $\lr{\hskip 0.6mm,}$ is the obtained double Poisson bracket on $A$, in the paragraph including \eqref{eq:closed} we have showed that $(\operatorname{SI}(3))$ for $\{m_{\bullet}\}_{\bullet \in \NN}$ is indeed equivalent to the fact that $\lr{\hskip 0.6mm,}$ is a closed morphism of dg vector spaces.
Similarly, the equivalent version of the Leibniz property given by \eqref{eq:leib2} shows that the latter is in fact tantamount to the vanishing of $\operatorname{SI}(4)_{\gan}|_{A\otimes A \otimes tA^{\#} \otimes A \otimes tA^{\#}}$, where write $tA^{\#}$ 
instead of $A^{\#}[d-1]$. 
Finally, the double Jacobi identity expressed by \eqref{eq:Jacobi-goal} shows that it is in fact equivalent to the vanishing of $\operatorname{SI}(5)_{\gan}|_{A \otimes tA^{\#} \otimes A \otimes tA^{\#} \otimes A\otimes tA^{\#}}$. 
Moreover, we remark that, since $A$ is a dg algebra, it is easy to see that the family of Stasheff identities \eqref{eq:ainftyalgebra} for the multiplications $\{m_{\bullet}\}_{\bullet \in \NN}$ on $\partial_{d-1}A$ is equivalent to just \eqref{eq:ainftyalgebra} for $n \in \{ 3, 4, 5\}$, since $(\operatorname{SI}(1))$ is equivalent to $\partial_{A} \circ \partial_{A}=0$, $(\operatorname{SI}(2))$ is the Leibniz property of $\partial_{A}$ with respect to the product $\mu_{A}$, and \eqref{eq:ainftyalgebra} trivially vanishes for $n > 5$. 

We will finally show that \eqref{eq:bij} is bijective. 
In order to do so, we first note that, given any good, small and fully manageable $d$-pre-Calabi-Yau structure $\{m_{\bullet}\}_{\bullet \in \NN}$ on $A$, it is uniquely determined by $m_{3}|_{A \otimes A^{\#}[d-1] \otimes A}$. 
Indeed, the fact that the $d$-pre-Calabi-Yau structure on $A$ is good tells us that the full $m_{3}$ on $\partial_{d-1}A$ 
is unique, the manageability hypothesis implies that $m_{1}$ and $m_{2}$ are uniquely determined by the 
dg algebra structure of $A$, whereas the smallness assumption tells us $m_{i} = 0$, for all $i > 3$. 
As a consequence, and using that the identity \eqref{eq:m3cor} implies that the associated double bracket $\lr{\hskip 0.6mm,}$ completely determines $m_{3}|_{A \otimes A^{\#}[d-1] \otimes A}$, we conclude that \eqref{eq:bij} is injective. 

We will finally show that it is surjective. 
By the comments in the previous paragraph, it suffices to show that, given any morphism $m_{3} : \partial_{d-1}A^{\otimes 3} \rightarrow \partial_{d-1}A$ of degree $-1$ on the dg algebra $\partial_{d-1}A$ described at the beginning of  Subsection \ref{subsection:main1}, whose product and differential are denoted by $m_{2}$ and $m_{1}$, respectively, satisfying that $m_{3}|_{T_{3,b}} = 0$,  
$m_{3}(A \otimes A^{\#}[d-1] \otimes A) \subseteq A$, $m_{3}(A^{\#}[d-1] \otimes A \otimes A^{\#}[d-1]) \subseteq A^{\#}[d-1]$, and the cyclic identities \eqref{eq:ip1}, for the natural bilinear form $\gan$ of degree $d - 1$, then the vanishing of 
$\operatorname{SI}(4)_{\gan}|_{A\otimes A \otimes tA^{\#} \otimes A \otimes tA^{\#}}$ is equivalent to 
$\operatorname{SI}(4)_{\gan} = 0$; 
furthermore, $\operatorname{SI}(5)_{\gan}|_{A \otimes tA^{\#} \otimes A \otimes tA^{\#} \otimes A\otimes tA^{\#}}=0$ 
is tantamount to the vanishing of $\operatorname{SI}(5)_{\gan}$. 
We leave the reader the tedious but straightforward verification that $\operatorname{SI}(4)_{\gan}|_{A\otimes A \otimes tA^{\#} \otimes A \otimes tA^{\#}} = 0$ is equivalent to $\operatorname{SI}(4)_{\gan}|_{\sigma(A\otimes A \otimes tA^{\#} \otimes A \otimes tA^{\#})}$, where $\sigma \in C_{5} \subseteq \mathbb{S}_{5}$ is any cyclic permutation, whereas 
$\operatorname{SI}(4)_{\gan}|_{\sigma(A\otimes A \otimes tA^{\#} \otimes A \otimes tA^{\#})}$ trivially vanishes if 
$\sigma \in \mathbb{S}_{5} \setminus C_{5}$.
Analogously, it is long but easy to verify that 
$\operatorname{SI}(5)_{\gan}|_{A \otimes tA^{\#} \otimes A \otimes tA^{\#} \otimes A \otimes tA^{\#}} = 0$ is equivalent to $\operatorname{SI}(5)_{\gan}|_{\sigma(A \otimes tA^{\#} \otimes A \otimes tA^{\#} \otimes A \otimes tA^{\#})}$, for any cyclic permutation $\sigma \in C_{6} \subseteq \mathbb{S}_{6}$, and $\operatorname{SI}(5)_{\gan}|_{\sigma(A \otimes tA^{\#} \otimes A \otimes tA^{\#} \otimes A \otimes tA^{\#})}$ is trivially zero if $\sigma \in \mathbb{S}_{6} \setminus C_{6}$.
This concludes the proof of the theorem. 
\end{proof}

\subsection{Relation between morphisms}
\label{subsection:main2}

Let $(A, \mu_{A},\partial_{A})$ and $(B, \mu_{B},\partial_{B})$ be two double Poisson dg algebras, with brackets $\lr{ \hskip 0.6mm,}_{A}$ and $\lr{ \hskip 0.6mm,}_{B}$ of degree $d$, respectively. 
A \emph{morphism of double Poisson dg algebras} $\phi : A \rightarrow B$ is a morphism of dg algebras 
satisfying that $(\phi \otimes \phi) \circ \lr{ \hskip 0.6mm,}_{A} = \lr{ \hskip 0.6mm,}_{B} \circ (\phi[d] \otimes \phi[d])$. 
Since $\phi : A \rightarrow B$ is a morphism of dg algebras, $B$ is a dg bimodule over $A$, so 
\[
 \partial_{d-1}\phi := A \oplus B^{\#}[d-1]
\]
has a dg algebra structure, as explained in the first two paragraphs of Subsection \ref{subsection:main1}. 

Moreover, $\partial_{d-1}\phi $ is naturally endowed with a super symmetric bilinear form 
\[
 \gan_{\phi} : (\partial_{d-1}\phi)^{\otimes 2} \rightarrow \kk
\]
of degree $d-1$ given by 
\begin{equation}
\label{eq:ganAB}
     \gan_{\phi}(tf,a) = (-1)^{|a|(|f|-d+1)}\gan_{\phi}(a,tf) = f(\phi(a)) \text{ and }
     \gan_{\phi}(a,b) = \gan_{\phi}(tf, tg) = 0,   
\end{equation}
for all homogeneous $a,b \in A$ and $f,g\in B^{\#}$.  

\begin{lemma}
\label{lemma:mor1}
Let $(A, \mu_{A},\partial_{A})$ and $(B, \mu_{B},\partial_{B})$ be two dg algebras and let $\phi : A \rightarrow B$ be a morphism of dg algebras. 
Consider the dg algebra structure on $\partial_{d-1}\phi = A \oplus B^{\#}[d-1]$ recalled before, whose product and differential will be denoted by $m_{2}^{\phi}$ and $m_{1}^{\phi}$, respectively. 
Then, $\partial_{d-1}\phi$ provided with $\gan_{\phi}$ defined in \eqref{eq:ganAB} is a degenerate $d$-cyclic
dg algebra.
\end{lemma}
\begin{proof}
We first remark that, by definition, $\partial_{d-1}\phi$ provided with $\gan_{\phi}$ is a degenerate $d$-cyclic dg algebra if and only if \eqref{eq:ip1} is verified for $n = 1, 2$, \textit{i.e.} 
\begin{equation}
\label{eq:dgcycmor}
\begin{split}
\gan_{\phi}\big(m_{1}^{\phi}(x),y\big) &= -(-1)^{|x| |y|} \gan_{\phi}\big(m_{1}^{\phi}(y),x\big),
\\
\gan_{\phi}\big(m_{2}^{\phi}(x,y),z\big) &= (-1)^{|x| (|y|+|z|)} \gan_{\phi}\big(m_{2}^{\phi}(y,z),x\big),
\end{split}
\end{equation}
for all homogeneous $x, y, z \in \partial_{d-1}\phi$. 

The first equation is trivially verified for $x, y \in A$ or $x, y \in B^{\#}[d-1]$, and using a symmetry argument it suffices to consider the case $x = a \in A$ and $y = tf$, with $f \in B^{\#}$, \textit{i.e.} $\gan_{\phi}(\partial_{A}(a),tf) = (-1)^{|tf| (|a| + 1)} \gan_{\phi}(t(f\circ \partial_{B}),a)$, which is equivalent to $f \circ \phi \circ \partial_{A}(a) = f \circ \partial_{B} \circ \phi (a)$. 
Since $\phi$ is a morphism of dg algebras, we conclude that \eqref{eq:ip1} for $n = 1$ is always verified. 

The definition of $\gan_{\phi}$ tells us that the second equation in \eqref{eq:dgcycmor} trivially holds if $x, y, z \in A$, 
or if there are at least two arguments among $x, y, z$ that belong to $B^{\#}[d-1]$. 
Finally, the three cases where two arguments of \eqref{eq:dgcycmor} are elements of $A$ and the other is in $B^{\#}[d-1]$ are clearly equivalent to the identity $\gan_{\phi}(a b, tf) = (-1)^{|a| (|b|+|tf|)} \gan_{\phi}(b \cdot tf,a) = (-1)^{|tf| (|a|+|b|)} \gan_{\phi}(b  tf \cdot a , b)$, for all homogeneous $a, b \in A$ and $f \in B^{\#}$. 
The latter is tantamount to $f \circ \phi (a b) = f(\phi(a) \phi(b))$, which is trivially verified for $\phi$, since it is a morphism of dg algebras. 
\end{proof}

Remarkably, the construction provided in Theorem \ref{theorem:main1} is functorial in the following sense:
\begin{theorem}
\label{theorem:main2}
Let $d \in \ZZ$, and let $(A, \mu_{A},\partial_{A})$ and $(B, \mu_{B},\partial_{B})$ be two double Poisson dg algebras, with brackets $\lr{ \hskip 0.6mm,}_{A}$ and $\lr{ \hskip 0.6mm,}_{B}$ of degree $d$, respectively. 
Let $\phi : A \rightarrow B$ be a morphism of double Poisson dg algebras. 
By Theorem \ref{theorem:main1}, $\partial_{d-1}A$ and $\partial_{d-1}B$ are provided with the corresponding cyclic $A_{\infty}$-algebra structures $\{ m_{\bullet}^{A} \}_{\bullet \in \NN}$ and $\{ m_{\bullet}^{B} \}_{\bullet \in \NN}$, respectively. 
Consider the dg algebra $\partial_{d-1}\phi = A \oplus B^{\#}[d-1]$ described previously, and define the unique good map 
$m_{3}^{\phi} : (\partial_{d-1}\phi)^{\otimes 3} \rightarrow \partial_{d-1}\phi$ satisfying that 
\begin{equation}
   \label{eq:m3-morphism-thm}
   m_{3}^{\phi}(a,tf,b) = m_{3}^{A}\big(a,t(f \circ \phi),b\big), \text{ and } m_{3}^{\phi}(tf,b,tg) = m_{3}^{B}\big(tf, \phi(b),tg\big),
\end{equation}
for all homogeneous $a, b \in A$ and $f, g \in B^{\#}$. 

Then, $\partial_{d-1}\phi$ is a fully manageable nice degenerate $d$-cyclic $A_{\infty}$-algebra, such that 
the maps $\Phi_{A} : \partial_{d-1}\phi \rightarrow \partial_{d-1}A$ and $\Phi_{B} : \partial_{d-1}\phi \rightarrow \partial_{d-1}B$ defined by $(a,tf) \mapsto (a,t(f \circ \phi))$ and $(a,tf) \mapsto (\phi(a),tf)$ for all $a \in A$ and $f \in B^{\#}$, respectively, 
are strict morphisms of $A_{\infty}$-algebras preserving the corresponding bilinear forms. 
\end{theorem}
\begin{proof}
We first remark that the fact that $\lr{\hskip 0.4mm,}_{B} \circ (\phi[d] \otimes \phi[d]) = (\phi \otimes \phi) \circ \lr{\hskip 0.4mm,}_{A}$ implies that 
\begin{equation}
\label{eq:phiequiv}
\begin{split}
     \phi\Big(m_{3}^{A}\big(a,t(f\circ \phi),b\big)\Big) &= m_{3}^{B}\big(\phi(a),tf,\phi(b)\big),  
     \\ 
     t^{-1}m_{3}^{A}\big(t(f\circ \phi),a,t(g \circ \phi)\big) &= \Big(t^{-1}m_{3}^{B}\big(tf,\phi(a),tg\big)\Big) \circ \phi,
\end{split}
\end{equation}
for all homogeneous $a, b \in A$ and $f, g \in B^{\#}$. 
Indeed, the first identity follows from 
\begin{align*}
&(g \circ \phi) \Big(m_{3}^{A}\big(a,t(f\circ \phi),b\big)\Big) = (-1)^{|tg|(|tf|+|a|+|b|)} \gan_{A}\Big(m_{3}^{A}\big(a,t(f\circ \phi),b\big) , t(g \circ \phi) \Big)  
\\
&= (-1)^{|a|(|f|+|b|+1) + |tg|(|tf|+|a|+|b|)} \big((g \circ \phi) \otimes (f \circ \phi)\big)\big(\lr{b,a}_{A}\big) 
\\
&= (-1)^{|a|(|f|+|b|+1) + |tg|(|tf|+|a|+|b|)} (g\otimes f)\big(\lr{\phi(b),\phi(a)}_{B}\big) 
\\
&= (-1)^{|tg|(|tf|+|a|+|b|)} \gan_{B}\Big(m_{3}^{B}\big(\phi(a),tf,\phi(b)\big) , tg \Big)
= g\Big(m_{3}^{B}\big(\phi(a),tf,\phi(b)\big)\Big),
\end{align*} 
whereas the second follows from
\begin{align*}
t^{-1}m_{3}^{A}\big(&t(f\circ \phi),a,t(g \circ \phi)\big)(b)
= \gan_{A}\Big(m_{3}^{A}\big(t(f\circ \phi),a,t(g \circ \phi)\big), b \Big)  
\\
&= - (-1)^{|tf|(|tg|+|a|+|b|)}\gan_{A}\Big(m_{3}^{A}\big(a,t(g \circ \phi),b\big), t(f\circ \phi) \Big)
\\
&= - (-1)^{|a|(|g|+|b|+1) + |tf|(|tg|+|a|+|b|)} \big((f \circ \phi) \otimes (g \circ \phi)\big)\big(\lr{b,a}_{A}\big) 
\\
&= - (-1)^{|a|(|g|+|b|+1) + |tf|(|tg|+|a|+|b|)} (f\otimes g)\big(\lr{\phi(b),\phi(a)}_{B}\big) 
\\
&= - (-1)^{|tf|(|tg|+|a|+|b|)} \gan_{B}\Big(m_{3}^{B}\big(\phi(a),tg,\phi(b)\big) , tf \Big)
\\
&= \gan_{B}\Big(m_{3}^{B}\big(tf,\phi(a),tg\big) , \phi(b) \Big)
= \Big(t^{-1}m_{3}^{B}\big(tf,\phi(a),tg\big)\Big)\big(\phi(b)\big).
\end{align*} 

We now show that $\partial_{d-1}\phi$ is an $A_{\infty}$-algebra. 
The first two Stasheff identities \eqref{eq:ainftyalgebra} are clearly verified, since they only involve the differential and the product of the dg algebras $A$ and $B$.
Moreover, the Stasheff identity \eqref{eq:ainftyalgebra} for $n = 3$ is also trivially verified. 
Indeed, since $\partial_{d-1}\phi$ provided with $m_{1}^{\phi}$ and $m_{2}^{\phi}$ is a dg algebra, 
it suffices to show that the contribution of the terms involving $m_{3}^{\phi}$ in the Stasheff identity for $n = 3$ at an element $x_{1} \otimes x_{2} \otimes x_{3}$, where $x_{i} \in A$ or $x_{i} \in B^{\#}[d-1]$ are homogeneous elements, vanish. 
It is easy to see that in this case the Stasheff identity for $n = 3$ of $\partial_{d-1}\phi$ is a consequence of the 
corresponding Stasheff identity for $n = 3$ of either $\partial_{d-1}A$ or $\partial_{d-1}B$. 

We now prove the Stasheff identity \eqref{eq:ainftyalgebra} for $n = 4$ at an element $x_{1} \otimes \dots \otimes x_{4}$, where $x_{i} \in A$ or $x_{i} \in B^{\#}[d-1]$ are homogeneous elements. 
It is easy to see that the only cases where there is at least one possibly nonvanishing term 
are the following 
\begin{enumerate}[label={\textup{(\alph*)}}]
\setcounter{enumi}{0} 
\item\label{item:a} $x_{2} \in B^{\#}[d-1]$ and $x_{1}, x_{3}, x_{4} \in A$;

\item\label{item:b} $x_{3} \in B^{\#}[d-1]$ and $x_{1}, x_{2}, x_{4} \in A$;

\item\label{item:c} $x_{1}, x_{3} \in A$ and $x_{2}, x_{4} \in B^{\#}[d-1]$;

\item\label{item:d} $x_{2}, x_{4} \in A$ and $x_{1}, x_{3} \in B^{\#}[d-1]$.
\end{enumerate}
The cases \ref{item:a} and \ref{item:b} only involve $m_{3}^{A}$ and $m_{3}^{B}$, respectively, so they follow from 
the Stasheff identity \eqref{eq:ainftyalgebra} for $n = 4$ of $\partial_{d-1}A$ and $\partial_{d-1}B$, respectively. 
The proof for cases \ref{item:c} and \ref{item:d} uses the precise same arguments, so we focus on the latter. 
We write $x_{1} = tf$, $x_{3} = tg$ with $f, g \in B^{\#}$, and $x_{2} = a$ and $x_{4} = b$. 
It is easy to see that the fourth Stasheff identity evaluated at $tf \otimes a \otimes tg \otimes b$ is precisely 
\begin{equation}
\label{eq:int}
     m_{3}^{B}\big(tf,\phi(a),tg.\phi(b)\big) - m_{3}^{B}\big(tf,\phi(a),tg\big).\phi(b) - (-1)^{|tf|} tf.\phi\Big(m_{3}^{A}\big(a,t(g\circ \phi),b\big)\Big)=0,     
\end{equation} 
or, equivalently,
\begin{equation}
\label{eq:int2}
     m_{3}^{B}\big(tf,\phi(a),tg.\phi(b)\big) - m_{3}^{B}\big(tf,\phi(a),tg\big).\phi(b) - (-1)^{|tf|} tf.m_{3}^{B}\big(\phi(a),tg,\phi(b)\big)=0,     
\end{equation} 
where we have used the first identity of \eqref{eq:phiequiv} in the last term of \eqref{eq:int}. 
Since \eqref{eq:int2} is precisely a particular instance of the Stasheff identity for $n = 4$ of $\partial_{d-1}B$, the Stasheff identity for $n = 4$ of $\partial_{d-1}\phi$ in the case \ref{item:c} follows. 

We finally prove the Stasheff identity \eqref{eq:ainftyalgebra} for $n = 5$ at an element $x_{1} \otimes \dots \otimes x_{5}$, where $x_{i} \in A$ or $x_{i} \in B^{\#}[d-1]$ are homogeneous elements. 
It is easy to see that the only cases where there are possibly nonvanishing terms are either if $x_{2}, x_{4} \in B^{\#}[d-1]$ and $x_{1}, x_{3}, x_{5} \in A$, or if $x_{1}, x_{3}, x_{5} \in B^{\#}[d-1]$ and $x_{2}, x_{4} \in A$. 
Since the arguments for both cases are the same, we only consider the former case, for which we write 
$x_{2} = tf$, $x_{4} = tg$, with $f, g \in B^{\#}[d-1]$, and $x_{1} =a$, $x_{3} = b$, $x_{5} = c$.
The corresponding Stasheff identity then reads 
 \begin{equation}
\label{eq:int3}
\begin{split}
     &m_{3}^{A}\Big(m_{3}^{A}\big(a,t(f \circ \phi), b \big),t(g \circ \phi),c \Big) 
- (-1)^{|a|} m_{3}^{A}\bigg(a, t\Big(t^{-1}m_{3}^{B}\big(tf, \phi(b),tg\big)\Big) \circ \phi,c \bigg)     
     \\
     &+ (-1)^{|a|+|tf|} m_{3}^{A}\Big(a,t(f \circ \phi) , m_{3}^{A}\big(b,t(g \circ \phi),c \big)\Big)=0,     
\end{split}
\end{equation} 
or, equivalently,
\begin{equation}
\label{eq:int4}
\begin{split}
     &m_{3}^{A}\Big(m_{3}^{A}\big(a,t(f \circ \phi), b \big),t(g \circ \phi),c \Big) 
- (-1)^{|a|} m_{3}^{A}\Big(a, m_{3}^{A}\big(t(f \circ \phi), b,t(g \circ \phi)\big),c \Big)     
     \\
     &+ (-1)^{|a|+|tf|} m_{3}^{A}\Big(a,t(f \circ \phi) , m_{3}^{A}\big(b,t(g \circ \phi),c \big)\Big)=0,     
\end{split}
\end{equation} 
where we have used the second identity of \eqref{eq:phiequiv} in the second term of \eqref{eq:int3}. 
Since \eqref{eq:int4} is precisely a particular case of the Stasheff identity of $\partial_{d-1}A$ for $n = 5$, the Stasheff identity of $\partial_{d-1}\phi$ for $n = 5$ follows. 
Taking into account that the Stasheff identities for $n \geq 6$ trivially vanish, because the higher products $m_{n}^{\phi}$ of $\partial_{d-1}\phi$ vanish for $n \geq 4$, we conclude that $\partial_{d-1}\phi$ is an $A_{\infty}$-algebra. 
Moreover, $\partial_{d-1}\phi$ is by definition small, and it is clearly fully manageable with respect to the dg algebra structure fixed at the beginning of this subsection.

We now show that $\gan_{\phi}$ satisfies the cyclicity property \eqref{eq:ip1} with respect to $m_{3}^{\phi}$, \textit{i.e.} 
\begin{equation*}
     \gan_{\phi}\big(m_{3}^{\phi}(x,y,z),w\big) = - (-1)^{|w|(|x|+|y|+|z|) } \gan_{\phi}\big(m_{3}^{\phi}(w,x,y),z\big),  
\end{equation*}
for all homogeneous $w, x, y, z \in \partial_{d-1}\phi$. 
It is easy to see that the only nontrivial cases are either if $x, z \in A$ and $w, y \in B^{\#}[d-1]$, or if $x, z \in B^{\#}[d-1]$ and $w, y \in A$, both of which are tantamount to 
\begin{equation}
\label{eq:ip1bisbis}
     \gan_{\phi}\big(m_{3}^{\phi}(a,tf,b),tg\big) = - (-1)^{|tg|(|a|+|b|+|tf|)} \gan_{\phi}\big(m_{3}^{\phi}(tg,a,tf),b\big),  
\end{equation}
for all homogeneous $a, b \in A$ and $f, g \in B^{\#}$. 
By definition, the left member of \eqref{eq:ip1bisbis} is precisely 
$(-1)^{|a|(|b|+|f|+1)} ((g \circ \phi) \otimes (f \circ \phi))(\lr{b,a})$, 
whereas the right member is 
\begin{align*}
     - (-1)^{|tg|(|a|+|b|+|tf|)} \gan_{\phi}\Big(&m_{3}^{B}\big(tg,\phi(a),tf\big),\phi(b)\Big) 
= \gan_{\phi}\Big(m_{3}^{B}\big(\phi(a),tf,\phi(b)\big), tg\Big) 
\\ 
&= (-1)^{|a|(|b|+|f|+1)} (g \otimes f)\big(\lr{\phi(b),\phi(a)}\big).     
\end{align*}
They clearly coincide, since $\phi$ is a morphism of double Poisson dg algebras. 
Combining this with the previous lemma we see that $\gan_{\phi}$ satisfies the cyclicity property \eqref{eq:ip1} with respect to $m_{n}^{\phi}$, for all $n \in \NN$. 

It is easy to verify that $\Phi_{A}$ and $\Phi_{B}$ commute with the corresponding differentials and the corresponding products, 
since $\phi$ is a morphism of dg algebras. 
To prove that they are strict morphisms $A_{\infty}$-algebras, it suffices to show that $\Phi_{A} \circ m_{3}^{\phi} = m_{3}^{A} \circ \Phi_{A}^{\otimes 3}$ and $\Phi_{B} \circ m_{3}^{\phi} = m_{3}^{B} \circ \Phi_{B}^{\otimes 3}$. 
Since both identities are proved by the same arguments, we will only consider the former, evaluated at $x_{1} \otimes x_{2} \otimes x_{3}$, for homogeneous $x_{1}, x_{2}, x_{3} \in \partial_{d-1}\phi$. 
By definition of $m_{3}^{\phi}$ and $m_{3}^{A}$, the only nontrivial cases are either when $x_{1}, x_{3} \in A$ and $x_{2} \in B^{\#}[d-1]$, or $x_{1}, x_{3} \in B^{\#}[d-1]$ and $x_{2} \in A$. 
The first case is direct from the definition of $\Phi_{A}$, whereas the second follows from the second identity in \eqref{eq:phiequiv}. 

It remains to show that $\Phi_{A}$ and $\Phi_{B}$ commute with the corresponding bilinear forms, \textit{i.e.} 
$\gan_{\phi} = \gan_{A} \circ \Phi_{A}^{\otimes 2} = \gan_{B} \circ \Phi_{B}^{\otimes 2}$, but this is straightforward. 
The theorem is thus proved.
\end{proof}

A direct consequence of the previous theorem is the following result. 
\begin{corollary}
\label{corollary:main2}
Assume the same hypotheses as in Theorem \ref{theorem:main2}. 
If the morphism of double Poisson dg algebras $\phi : A \rightarrow B$ is a quasi-isomorphism and $B$ is locally finite dimensional, 
then the strict morphisms of degenerate cyclic $A_{\infty}$-algebras $\Phi_{A} : \partial_{d-1}\phi \rightarrow \partial_{d-1}A$ and $\Phi_{B} : \partial_{d-1}\phi \rightarrow \partial_{d-1}B$ are also quasi-isomorphisms.
\end{corollary}

Motivated by the previous theorem, we introduce the following. 
\begin{definition}
\label{definition:morprecy}
Let $A$ and $A'$ be two $d$-pre-Calabi-Yau algebra structures on the graded vector spaces $A$ and $A'$, respectively. 
A \emph{morphism} from $A$ to $A'$ is a triple $(C,\Phi,\Psi)$, where $C$ is a degenerate $(d-1)$-ultracyclic $A_{\infty}$-algebra, and $\Phi : C \rightarrow A$ and $\Psi : C \rightarrow A'$ are strict morphisms of $A_{\infty}$-algebras that preserve the 
corresponding bilinear forms. 

We say that a morphism $(C,\Phi,\Psi)$ from $A$ to $A'$ and a morphism $(C',\Phi',\Psi')$ from $A'$ to $A''$ are \emph{composable} if there exists a triple $(C'',\Phi'',\Psi'')$, where $C''$ is a degenerate $(d-1)$-ultracyclic $A_{\infty}$-algebra, $\Phi'' : C'' \rightarrow C$ and $\Psi'' : C'' \rightarrow C'$ are strict morphisms of $A_{\infty}$-algebras that preserve the 
corresponding bilinear forms and $\Psi \circ \Phi'' = \Phi' \circ \Psi''$. 
The \emph{composition} of $(C,\Phi,\Psi)$ and $(C',\Phi',\Psi')$ is then defined to be $(C'',\Phi \circ \Phi'',\Psi' \circ \Psi')$. 
\end{definition}

The proof of the following result follows exactly the same pattern as the (last part of the proof of) Theorem \ref{theorem:main2}, so we leave it to the reader. 
\begin{theorem}
\label{theorem:main2bis}
Let $d \in \ZZ$, and let $(A, \mu_{A},\partial_{A})$, $(B, \mu_{B},\partial_{B})$ and $(C, \mu_{C},\partial_{C})$ be three double Poisson dg algebras, with brackets $\lr{ \hskip 0.6mm,}_{A}$, $\lr{ \hskip 0.6mm,}_{B}$ and $\lr{ \hskip 0.6mm,}_{C}$ of degree $d$, respectively. 
Let $\phi : A \rightarrow B$ and $\psi : B \rightarrow C$ be two morphisms of double Poisson dg algebras, and let $\upsilon = \psi \circ \phi$. 
Following Theorem \ref{theorem:main2}, consider the morphisms $(\partial_{d-1}\phi , \Phi_{A}, \Phi_{B})$ and $(\partial_{d-1}\psi , \Psi_{A}, \Psi_{B})$ induced by $\phi$ and $\psi$, respectively. 

Consider the fully manageable nice degenerate $d$-cyclic $A_{\infty}$-algebra $\partial_{d-1}\upsilon$ on 
$A \oplus C^{\#}[d-1]$. 
Then the maps $\Upsilon_{\phi} : \partial_{d-1}\upsilon \rightarrow \partial_{d-1}\phi$ and $\Upsilon_{\psi} : \partial_{d-1}\upsilon \rightarrow \partial_{d-1}\psi$ defined by $(a,tf) \mapsto (a,t(f \circ \psi))$ and $(a,tf) \mapsto (\phi(a),tf)$ for all $a \in A$ and $f \in C^{\#}$, respectively, 
are strict morphisms of $A_{\infty}$-algebras preserving the corresponding bilinear forms, and satisfying that $\Phi_{B} \circ \Upsilon_{\phi} = \Psi_{A} \circ \Upsilon_{\psi}$. 
As a consequence, $(\partial_{d-1}\phi , \Phi_{A}, \Phi_{B})$ and $(\partial_{d-1}\psi , \Psi_{A}, \Psi_{B})$ are composable morphisms and their composition is $(\partial_{d-1}\upsilon, \Phi_{A} \circ \Upsilon_{\phi}, \Psi_{B} \circ \Upsilon_{\psi})$.
\end{theorem}

This result tells us that the constructions in Theorems \ref{theorem:main1} and \ref{theorem:main2} define a (partial) functor from the category of $d$-double Poisson dg algebras to the partial category of $d$-pre-Calabi-Yau algebras provided with the morphisms introduced in Definition \ref{definition:morprecy}, that preserves quasi-isomorphisms, under some mild assumptions.  

\section{\texorpdfstring{Pre-Calabi-Yau structures and double $P_{\infty}$-algebras}{Pre-Calabi-Yau structures and double P-infinity-algebras}}
\label{section:P-infinity}

We now introduce the definition of a double $P_{\infty}$-algebra. 
It is essentially the same as the one presented in \cite{Sch09}, Def. 4.1, up to some sign differences. 
\begin{definition}
\label{definition:poinf}
A \emph{double $P_\infty$-algebra} is a (nonunitary) graded algebra $A=\oplus_{n\in\mathbb{Z}} A^{n}$ provided with a family 
of homogeneous maps $\lr{\dots}_{p} : A^{\otimes p} \rightarrow A^{\otimes p}$ indexed by $p \in \NN$, 
where $\lr{\dots}_{p}$ has degree $2-p$, satisfying that 
\begin{enumerate}[label={\textup{(\roman*)}}]
\setcounter{enumi}{0} 
\item\label{item:dpainf1} $\tau_{A,p}(\sigma) \circ \lr{\dots}_{p} \circ \tau_{A,p}(\sigma^{-1}) = \operatorname{sgn}(\sigma) \lr{\dots}_{p}$, for all $\sigma \in \mathbb{S}_{p}$;
\item\label{item:dpainf2} for all $p \in \NN$ and homogeneous elements $a_{1}, \dots, a_{p-1} \in A$, the homogeneous map 
\[     \operatorname{AD}(a_{1}, \dots, a_{p-1}) : A \rightarrow A \otimes A     \] 
of degree $|a_{1}|+\dots+|a_{p-1}|+2-p$ given by $a \mapsto \lr{a_{1}, \dots, a_{p-1},a}_{p}$ is a \emph{double derivation} of $A$, \textit{i.e.}
\begin{equation}
\label{eq:dpoisinfpoi}
\tag{$\operatorname{DLeib}_{\infty}(p)$}
\begin{split}
     &\lr{a_{1}, \dots, a_{p-1},a b}_{p} = \lr{a_{1}, \dots, a_{p-1},a}_{p} b
     \\
     & \hskip 3cm + (-1)^{|a|(p+\sum\limits_{j=1}^{p-1} |a_{j}|)} a \lr{a_{1}, \dots, a_{p-1},b}_{p},     
     \end{split}
     \end{equation}
for all homogeneous $a, b \in A$;
\item\label{item:dpainf3} for all $p \in \NN$, 
\begin{equation}
\label{eq:dpoisinfjac}
\tag{$\operatorname{DJac}_{\infty}(p)$}
\begin{split}
     \sum_{i=1}^{p} (-1)^{i(p+1)} \sum_{\sigma \in C_{p}} \operatorname{sgn}(\sigma) \tau_{A,p}(\sigma) \circ \lr{\dots}_{i,p-i+1} \circ \tau_{A,p}(\sigma^{-1}) = 0,
\end{split}
\end{equation}
where  
\[     \lr{\dots}_{i,p-i+1} = \big(\lr{\dots}_{i} \otimes \mathrm{id}_{A}^{\otimes (p-i)}\big) \circ \big(\mathrm{id}_{A}^{\otimes (i-1)} \otimes \lr{\dots}_{p-i+1}\big).     \]
\end{enumerate}
\end{definition}

Let $A$ be a double $P_\infty$-algebra with brackets $\lr{\dots}_{p} : A^{\otimes p} \rightarrow A^{\otimes p}$ for $p \in \NN$. 
Given $p \in \NN$ and $n \geq p$, we define $\lr{\dots}_{p,L} : A^{\otimes n} \rightarrow A^{\otimes n}$ 
as $\lr{\dots}_{p} \otimes \mathrm{id}_{A}^{\otimes (n-p)}$. 

\begin{remark}
We leave to the reader the straightforward verification that a double $P_\infty$-algebra $(A,\mu_{A})$ with brackets $\{\lr{\dots}_{p}\}_{p \in \NN}$ satisfying that $\lr{\dots}_{p} = 0$ for all $p >2$ is a double Poisson dg algebra of degree zero, with $\lr{\hskip 0.6mm,}_{A} = \lr{\dots}_{2}$ and $\partial_{A} = \lr{\dots}_{1}$. 
Indeed, $\operatorname{Jac}_{\infty}(1)$ means exactly that $\partial_{A}$ is a differential, $\operatorname{Jac}_{\infty}(2)$ 
is precisely the fact that $\lr{\hskip 0.6mm,}_{A}$ is a morphism of closed dg vector spaces, $\operatorname{Jac}_{\infty}(3)$ is the double Jacobi identity for $\lr{\hskip 0.6mm,}_{A}$, $\operatorname{Leib}_{\infty}(1)$ means exactly that $\partial_{A}$ is a derivation of the 
graded algebra $(A,\mu_{A})$ and $\operatorname{Leib}_{\infty}(2)$ is the Leibniz identity for $\lr{\hskip 0.6mm,}_{A}$. 
The antisymmetry conditions given in the previous definition and in Definition \ref{dg-Poisson-double} \ref{item:dpa1}
  are clearly equivalent. 
The identities \eqref{eq:dpoisinfpoi} for $p> 2$ and \eqref{eq:dpoisinfjac} for $p>3$ are trivially verified. 
\end{remark}

\begin{theorem}
\label{theorem:main3}
Let $A=\oplus_{n\in\mathbb{Z}} A^{n}$ be a (nonunitary) graded algebra with product $\mu_A$. 
Consider the graded algebra structure on $\partial_{-1}A = A\oplus A^\#[-1]$ described in the first two paragraphs of Subsection \ref{subsection:main1}, with product $m_{2}$, as well as the natural nondegenerate bilinear form of degree $-1$ given by \eqref{eq:natform}.
Then, given a good manageable special pre-Calabi-Yau structure $\{ m_{\bullet} \}_{\bullet\in\mathbb{N}}$ on $A$, 
we define the family of maps $\{ \lr{\dots}_{p} \}_{p \in \NN}$ with $\lr{\dots}_{p} : A^{\otimes p} \rightarrow A^{\otimes p}$ given by 
\begin{equation}
\label{eq:mncor}
(f_{1} \otimes \dots \otimes f_{p})\big(\lr{a_{1}, \dots, a_{p}}_{p}\big) = s_{f_{1}, \dots, f_{p}}^{a_{1}, \dots, a_{p}}
\,\gan\big(m_{2p-1}(a_{p},tf_{p},\dots,a_{2},tf_{2},a_{1}),tf_{1}\big),
\end{equation}
for $p \in \NN$ and all homogeneous $a_{1}, \dots, a_{p} \in A$ and $f_{1}, \dots, f_{p} \in A^{\#}$, where 
 \begin{equation}
\label{eq:mncor2}
\begin{split}
s_{f_{1}, \dots, f_{p}}^{a_{1}, \dots, a_{p}} =& (-1)^{|a_{p}| |f_{1}| + (p+1)(|a_{p}| + |f_{1}|) + \sum\limits_{j=1}^{p} (p-j) |a_{j}| + \sum\limits_{j=1}^{p} (j-1) |f_{j}|} 
\\
&(-1)^{\sum\limits_{1\leq i<j <p} |a_{i}| |a_{j}| + \sum\limits_{1< i<j \leq p} |f_{i}| |f_{j}| + \sum\limits_{1< i\leq j <p} |f_{i}| |a_{j}|}.
\end{split}
\end{equation}
Then, $\{ \lr{\dots}_{p} \}_{p \in \NN}$ determines a structure of a double $P_{\infty}$-algebra on the graded algebra $A$.
Moreover, the map
\begin{equation}
\label{eq:bij2}
 \bigg\{ \begin{matrix}
 \text{good manageable special}
 \\
 \text{pre-CY structures $\{m_{\bullet}\}_{\bullet \in \NN}$ on $A$}
 \end{matrix} 
\bigg\} 
\longrightarrow
 \bigg\{ \begin{matrix}
 \text{double $P_{\infty}$-algebra}
 \\
 \text{structures $\{ \lr{\dots}_{\bullet} \}_{\bullet \in \NN}$ on $A$}
 \end{matrix} 
\bigg\} 
\end{equation}
given by sending $\{ m_{\bullet} \}_{\bullet \in \NN}$ to the family of maps $\{ \lr{\dots}_{\bullet} \}_{\bullet \in \NN}$ determined by \eqref{eq:mncor} is a bijection. 
\end{theorem}
\begin{proof}
We will first prove that the family of brackets $\{ \lr{\dots}_{p} \}_{p \in \NN}$ defined by \eqref{eq:mncor} gives indeed a 
double $P_{\infty}$-algebra structure on the graded algebra $A$. 
In other words, we shall prove that this bracket satisfies the conditions of Definition \ref{definition:poinf}.
As explained in the first paragraph of the proof of Theorem \ref{theorem:main1}, we can assume without loss of generality 
that $\sum_{j=1}^{p} |a_{j}| +2-p= \sum_{j=1}^{p} |f_{j}|$ in \eqref{eq:mncor2}, else the identity \eqref{eq:mncor} trivially holds. 

We will first prove the antisymmetric condition \ref{item:dpainf1} given in Definition \ref{definition:poinf}, \textit{i.e.}
\begin{equation}
\label{eq:antiinf}
     (-1)^{\epsilon(\sigma,\bar{a})} \sigma \big(\lr{a_{\sigma(1)}, \dots, a_{\sigma(p)}}_{p}\big) = \operatorname{sgn}(\sigma) \lr{a_{1}, \dots, a_{p}}_{p},
\end{equation}     
for all homogeneous $a_{1}, \dots, a_{p} \in A$, where $\bar{a} = a_{1} \otimes \dots \otimes a_{p}$ and $\epsilon(\sigma,\bar{a})$ was defined in \eqref{eq:permeps2}.
Evaluating $f_{1}\otimes \dots \otimes f_{p}$ at both members of the previous equation, where $f_{1}, \dots, f_{p} \in A^{\#}$ are homogeneous, 
it is clear that \eqref{eq:antiinf} is equivalent to 
\begin{equation}
\label{eq:antiinf2}
\begin{split}
     (f_{1} \otimes \dots \otimes f_{p}) &\big(\sigma \lr{a_{\sigma(1)}, \dots, a_{\sigma(p)}}_{p}\big) 
     \\
     &= \operatorname{sgn}(\sigma) (-1)^{\epsilon(\sigma,\bar{a})} (f_{1} \otimes \dots \otimes f_{p}) \big(\lr{a_{1}, \dots, a_{p}}_{p}\big),
\end{split}
\end{equation} 
for all homogeneous $a_{1}, \dots, a_{p} \in A$ and $f_{1}, \dots, f_{p} \in A^{\#}$. 
Using \eqref{eq:unipermvecfun} on the left member as well as \eqref{eq:mncor} on each side, we obtain that \eqref{eq:antiinf2} is equivalent to 
\begin{equation}
\label{eq:antiinf3}
\begin{split}
     s_{f_{\sigma(1)}, \dots, f_{\sigma(p)}}^{a_{\sigma(1)}, \dots, a_{\sigma(p)}}  & (-1)^{\epsilon(\sigma,\bar{f})} \gan\big(m_{2p-1}(a_{\sigma(p)},tf_{\sigma(p)}, \dots,a_{\sigma(1)}),tf_{\sigma(1)}\big) 
     \\
     &= \operatorname{sgn}(\sigma) (-1)^{\epsilon(\sigma,\bar{a})} s_{f_{1}, \dots, f_{p}}^{a_{1}, \dots, a_{p}} \gan\big(m_{2p-1}(a_{p},tf_{p}, \dots,a_{1}),tf_{1}\big),
\end{split}
\end{equation} 
where $\bar{f} = f_{1} \otimes \dots \otimes f_{p}$. 
By the ultracyclicity property of $\gan$, the left member of the previous equation is precisely $\gan(m_{2p-1}(a_{p},tf_{p}, \dots,a_{1}),tf_{1})$ multiplied by
\begin{equation}
\label{eq:antiinf4}
   s_{f_{\sigma(1)}, \dots, f_{\sigma(p)}}^{a_{\sigma(1)}, \dots, a_{\sigma(p)}}  (-1)^{\epsilon(\sigma,\bar{f})} 
   (-1)^{\epsilon(\xan_{p}(\sigma),\overline{fa})},
\end{equation}
where $\overline{fa} = f_{1} \otimes a_{1} \otimes \dots \otimes f_{p} \otimes a_{p}$. 
Hence, comparing \eqref{eq:antiinf3} and \eqref{eq:antiinf4}, we see that \eqref{eq:antiinf} holds if and only if 
\begin{equation}
\label{eq:antiinf5}
\begin{split}
   s_{f_{\sigma(1)}, \dots, f_{\sigma(p)}}^{a_{\sigma(1)}, \dots, a_{\sigma(p)}}  = (-1)^{\epsilon(\sigma,\bar{f})} 
   (-1)^{\epsilon(\xan_{p}(\sigma),\overline{fa})} \operatorname{sgn}(\sigma) (-1)^{\epsilon(\sigma,\bar{a})} s_{f_{1}, \dots, f_{p}}^{a_{1}, \dots, a_{p}}.
\end{split}
\end{equation}
Replacing $s_{f_{1}, \dots, f_{p}}^{a_{1}, \dots, a_{p}}$ by its definition and considering the case where $\sigma$ is any transposition of two successive elements, it is easy but lengthy to show that the antisymmetric condition \eqref{eq:antiinf5} holds, which in turn implies that \eqref{eq:antiinf} holds, as was to be shown. 

We shall now prove the Leibniz identity given in Definition \ref{definition:poinf}, \ref{item:dpainf2} for a fixed $p \in \NN$.
In order to do so, let us consider the identity \eqref{eq:ainftyalgebra} of the $A_\infty$-algebra structure of $\partial_{-1}A$ for $n=2p$. 
Since the $A_{\infty}$-algebra structure on $\partial_{-1}A$ is essentially odd, \eqref{eq:ainftyalgebra} for $n=2p$ reduces to 
\eqref{eq:ainftyalgebraodd-ev}, which, evaluated at $a_{0}\otimes b_{0}\otimes tf_{1}\otimes a_{1}\otimes \dots \otimes tf_{p-1}\otimes a_{p-1}$, gives 
\begin{equation*}
\begin{split}
 &-(-1)^{|a_{0}|} a_{0}.m_{2p-1}(b_{0}, tf_{1}, a_{1}, \dots, tf_{p-1}, a_{p-1})
 \\
 +
 m_{2p-1}(a_{0} b_{0}, &tf_{1}, a_{1}, \dots, tf_{p-1}, a_{p-1}) 
 - m_{2p-1}(a_{0}, b_{0}.tf_{1}, a_{1}, \dots, tf_{p-1}, a_{p-1})=0,
\end{split}
\end{equation*}
where $a_{0}, b_{0}, a_{1}, \dots, a_{p-1} \in A$, and $f_{1}, \dots, f_{p-1}\in A^{\#}$ are homogeneous elements. 
Applying $\gan(-,tf_{p})$ to the previous equation, for an arbitrary homogeneous $f_{p}\in A^{\#}$, we get 
\begin{equation}
\label{eq:n-Leibniz-inicial}
\begin{split}
 &-(-1)^{|a_{0}|} \gan\big(a_{0}.m_{2p-1}(b_{0}, tf_{1}, a_{1}, \dots, tf_{p-1}, a_{p-1}),tf_{p}\big)
 \\
 &+
 \gan\big(m_{2p-1}(a_{0} b_{0}, tf_{1}, a_{1}, \dots, tf_{p-1}, a_{p-1}),tf_{p}\big)
 \\
 &- \gan\big(m_{2p-1}(a_{0}, b_{0}.tf_{1}, a_{1}, \dots, tf_{p-1}, a_{p-1}),tf_{p}\big)=0.
\end{split}
\end{equation}

By the cyclicity property of $\gan$, the identity $(tf_{p}).a_{0}=t(f_{p}.a_{0})$ as well as \eqref{eq:mncor}, 
we see that the first term in the left member of \eqref{eq:n-Leibniz-inicial} is 
\[     - s_{f_{p}a_{0},f_{p-1},\dots,f_{1}}^{a_{p-1},\dots,a_{1},b_{0}} (-1)^{\epsilon} (f_{p}.a_{0} \otimes f_{p-1} \otimes \dots \otimes f_{1})\big(\lr{a_{p-1},\dots,a_{1},b_{0}}_{p}\big),     \]
where 
\[     \epsilon = |a_{0}|(p+|b_{0}|+|f_{p}|+\sum\limits_{j=1}^{p-1}(|a_{j}|+|f_{j}|)).     \]
Taking into account the identity
\[
(f_{1}.a\otimes f_{2}\otimes \dots \otimes f_{m})(v_1\otimes \dots \otimes v_m)=(-1)^{|a|\sum\limits_{j=2}^{m}|f_{j}|}(f_{1}\otimes \dots \otimes f_{m})(a.v_1\otimes v_{2} \otimes \dots \otimes v_m),
\]
for all homogeneous $a \in A$, $v_{1}, \dots, v_{m} \in M$ and $f_{1}, \dots, f_{m} \in M^{\#}$, where $M$ is a graded $A$-bimodule, we conclude that the first term in the left member of \eqref{eq:n-Leibniz-inicial} is precisely 
\begin{equation}
\label{eq:primer-termino-leibniz}
- s_{f_{p}.a_{0},f_{p-1},\dots,f_{1}}^{a_{p-1},\dots,a_{1},b_{0}} (-1)^{|a_{0}|(p+|b_{0}|+|f_{p}|+\sum\limits_{j=1}^{p-1}|a_{j}|)}(f_{p} \otimes \dots \otimes f_{1})\big(a_{0}\lr{a_{p-1},\dots,a_{1},b_{0}}_{p}\big).
\end{equation}
On the other hand, using \eqref{eq:mncor}, we see that the second term of the left member of \eqref{eq:n-Leibniz-inicial} is precisely
\begin{equation}
 \label{eq:segundo-termino-leibniz}
 s_{f_{p},\dots,f_{1}}^{a_{p-1},\dots,a_{1},a_{0} b_{0}} (f_{p} \otimes \dots \otimes f_{1})\big(\lr{a_{p-1},\dots,a_{1},a_{0} b_{0}}_{p}\big).
\end{equation}
Similarly, by the identity $b_{0}.tf_{1}=(-1)^{|b_{0}|}t(b_0.f_{1})$ and \eqref{eq:mncor}, the third term of the left member of 
\eqref{eq:n-Leibniz-inicial} is 
\begin{equation}
\label{eq:tercer-termino-leibniz}
 - s_{f_{p},\dots,f_{2},b_{0}.f_{1}}^{a_{p-1},\dots,a_{0}} (-1)^{|b_{0}|(p+|b_{0}|+|f_{1}|+\sum\limits_{j=0}^{p-1}|a_{j}|)} (f_{p} \otimes \dots \otimes f_{1})\big(\lr{a_{p-1},\dots,a_{0}}_{p} b_{0}\big),
\end{equation}
where we have used that 
\[
(f_{1}\otimes \dots \otimes f_{m-1} \otimes a.f_{m})(v_1\otimes \dots \otimes v_m)=(-1)^{\epsilon'}(f_{1}\otimes \dots \otimes f_{m})(v_1\otimes \dots \otimes v_{m-1} \otimes v_{m}.a),
\]
for all homogeneous $a \in A$, $v_{1}, \dots, v_{m} \in M$ and $f_{1}, \dots, f_{m} \in M^{\#}$, where $M$ is a graded $A$-bimodule, and $\epsilon' = |a|(|f_{m}|+|\sum_{j=2}^{m}|v_{j}|)$. 

Replacing \eqref{eq:primer-termino-leibniz}, \eqref{eq:segundo-termino-leibniz} and \eqref{eq:tercer-termino-leibniz} into \eqref{eq:n-Leibniz-inicial} and comparing it with equation \eqref{eq:dpoisinfpoi}, we see that the latter holds if and only if 
\begin{equation}
\label{eq:leibinf}
\begin{split}
 s_{f_{p},\dots,f_{1}}^{a_{p-1},\dots,a_{1},a_{0} b_{0}}&=  s_{f_{p},\dots,f_{2},b_{0}.f_{1}}^{a_{p-1},\dots,a_{0}} (-1)^{|b_{0}|(p+|b_{0}|+|f_{1}|+\sum\limits_{j=0}^{p-1}|a_{j}|)}, 
 \\
 s_{f_{p},\dots,f_{1}}^{a_{p-1},\dots,a_{1},a_{0} b_{0}}&= (-1)^{|a_{0}|(|b_{0}|+|f_{p}|)} s_{f_{p}.a_{0},f_{p-1},\dots,f_{1}}^{a_{p-1},\dots,a_{1},b_{0}}.
\end{split}
\end{equation}
It is rather tedious but straightforward to check that our choice \eqref{eq:mncor2} satisfies the previous identities, so the Leibniz property is verified. 
 
\begin{remark}
As in Theorem \ref{theorem:main1}, assuming that $s_{f_{1},\dots,f_{p}}^{a_{1},\dots,a_{p}}$ is just a function of the degrees $|a_{1}|, \dots, |a_{p}|$ and $|f_{1}|, \dots, |f_{p}|$ (satisfying that $\sum_{j=1}^{p}|a_{j}| + |f_{j}| = p-2$ ($\operatorname{mod} 2$)), one can also show that our choice for $s_{f_{1},\dots,f_{p}}^{a_{1},\dots,a_{p}}$ is the unique solution of \eqref{eq:antiinf5} and \eqref{eq:leibinf}, up to a multiplicative constant $\pm 1$. 
This is again how we found such an involved expression. 
In fact, the uniqueness of such a solution (up to multiplicative constant) already holds if one considers \eqref{eq:antiinf5} for only cyclic permutations and \eqref{eq:leibinf}. 
\end{remark}

We will now prove \eqref{eq:dpoisinfjac} for $p \in \NN$. 
In order to do so, we consider \eqref{eq:ainftyalgebra} for $n = 2p-1$. 
Since the $A_{\infty}$-algebra structure on $\partial_{-1}A$ is essentially odd, it reduces to \eqref{eq:ainftyalgebraodd-odd}. 
Since $m_{2}$ is associative, the first term in the left member of \eqref{eq:ainftyalgebraodd-odd} vanishes, so it is equivalent to 
\begin{equation*}
\begin{split}
   \sum_{i=1}^{p} &\sum_{r=0}^{i-1} m_{2i-1} \circ \big(\mathrm{id}_{A}^{\otimes 2r} \otimes m_{2(p-i)+1} \otimes \mathrm{id}_{A}^{\otimes (2(i-1-r))}\big)  
   \\
   &+ \sum_{i=1}^{p} \sum_{r=0}^{i-2} m_{2i-1} \circ \big(\mathrm{id}_{A}^{\otimes (2r+1)} \otimes m_{2(p-i)+1} \otimes \mathrm{id}_{A}^{\otimes (2(i-1-r)-1)}\big)  = 0.
\end{split}
\end{equation*}
If we evaluate it at $a_{1} \otimes tf_{1} \otimes\dots \otimes a_{p-1}\otimes tf_{p-1} \otimes a_{p}$ and we apply $\gan(\place,tf_{p})$, for homogeneous $a_{1}, \dots, a_{p} \in A$ and $f_{1}, \dots, f_{p} \in A^{\#}$, it gives 
\begin{small}
\begin{equation}
\label{eq:ainftyalgebraodd-odd2}
\begin{split}
   &\sum_{i=1}^{p} \sum_{r=0}^{i-1} (-1)^{r+\sum\limits_{j=1}^{r}(|a_{j}|+|f_{j}|)} 
   \gan\Big(m_{2i-1} \big(a_{1}, tf_{1}, \dots, a_{r},tf_{r},  
   \\ 
   &\hskip 3cm m_{2(p-i)+1}(a_{r+1}, tf_{r+1}, \dots, a_{r+p-i+1}), tf_{r+p-i+1}, \dots,  a_{p}\big), tf_{p}\Big)
      \\
   &+ \sum_{i=1}^{p} \sum_{r=0}^{i-2} (-1)^{r+|a_{r+1}|+\sum\limits_{i=1}^{r}(|a_{i}|+|f_{i}|)} 
   \gan\Big(m_{2i-1} \big(a_{1}, tf_{1}, \dots, a_{r},tf_{r}, a_{r+1}, 
   \\
   &\hskip 3cm m_{2(p-i)+1}(tf_{r+1}, a_{r+2}, \dots, tf_{r+p-i+1}), a_{r+p-i+2}, \dots,  a_{p}\big), tf_{p}\Big) = 0.
\end{split}
\end{equation}
\end{small}

Using the cyclicity of $\gan$, the terms appearing in the first two lines of \eqref{eq:ainftyalgebraodd-odd2} can be rewritten as 
\begin{small}
\begin{equation}
\label{eq:ainftyalgebraodd-odd3-1}
\begin{split}
   &\sum_{i=1}^{p} \sum_{r=0}^{i-1} (-1)^{\alpha} \gan\Big(m_{2i-1} \big(m_{2(p-i)+1}(a_{r+1}, tf_{r+1}, \dots, a_{r+p-i+1}),
      \\
   & \hskip 1cm tf_{r+p-i+1}, \dots,  a_{p}, tf_{p}, a_{1}, tf_{1}, \dots, a_{r}\big),tf_{r}\Big),
\end{split}
\end{equation}
\end{small}
where $\alpha = (r+\sum_{j=1}^{r}(|a_{j}|+|f_{j}|))(p-r+\sum_{j=r+1}^{p}(|a_{j}|+|f_{j}|))$. 

Concerning the terms in the last two lines of \eqref{eq:ainftyalgebraodd-odd2}, we first use the cyclicity of $\gan$ to move $m_{2(p-i)+1}(tf_{r+1}, a_{r+2}, \dots, tf_{r+p-i+1})$ to the last argument of $\gan$. 
Then, we apply the super symmetry of $\gan$ to flip its two arguments, and then again the cyclicity of $\gan$. 
After these computations, the terms in the last two lines of the left member of \eqref{eq:ainftyalgebraodd-odd2} become
\begin{small}
\begin{equation}
\label{eq:ainftyalgebraodd-odd3-2}
\begin{split}
   &\sum_{i=1}^{p} \sum_{r=0}^{i-2} (-1)^{\beta} 
   \gan\Big(m_{2(p-i)+1} \big(m_{2i-1}(a_{r+p-i+2}, tf_{r+p-i+2}, \dots, a_{p},tf_{p},a_{1},tf_{1}, \dots, 
   \\
   & \hskip 2cm a_{r},tf_{r},a_{r+1}), tf_{r+1}, \dots,  a_{r+p-i+1}\big), tf_{r+p-i+1}\Big),
\end{split}
\end{equation}
\end{small}
where $\beta = (r+p-i+1+\sum_{j=1}^{r+p-i+1}(|a_{j}|+|f_{j}|))(i+r+1+\sum_{j=r+p-i+2}^{p}(|a_{j}|+|f_{j}|))$. 

Before proceeding further, we will provide the following useful result:
\begin{fact}
\label{fact:jacobi}
Let $a_{1}, \dots, a_{p} \in A$ and $f_{1}, \dots, f_{p} \in A^{\#}$ be homogeneous elements. 
Then, given any $i \in \{ 1, \dots, p\}$, 
\begin{equation}
\label{eq:relacion-lemma-tecnicon}
\begin{split}
  &{}^{i}\Box_{f_{1}, \dots, f_{p}}^{a_{1}, \dots, a_{p}} (f_{1} \otimes \dots \otimes f_{p})\big(\lr{a_{1}, \dots, a_{i-1},\lr{a_{i}, \dots, a_{p}}_{p-i+1} }_{i,L}\big)
  \\
  &=\gan\Big(m_{2i-1}\big(m_{2(p-i)+1}(a_{p},tf_{p}, \dots, a_{i+1},tf_{i+1},a_{i}),tf_{i},\dots, a_{2},tf_{2},a_{1}\big),tf_{1}\Big),
\end{split}
\end{equation}
where 
\begin{equation}
\label{eq:symbol-lemman}
\begin{split}
{}^{i}\Box_{f_{1}, \dots, f_{p}}^{a_{1}, \dots, a_{p}} =& (-1)^{(p+1)(i+1) + \sum\limits_{j=1}^{i-1} (i-j) |a_{j}| + \sum\limits_{j=i}^{p} (j-1) |a_{j}| + p |f_{1}| + \sum\limits_{j=2}^{p} (j-1) |f_{j}|} 
\\
&(-1)^{\sum\limits_{1\leq j<k <i} |a_{j}| |a_{k}|+\sum\limits_{i\leq j<k \leq p} |a_{j}| |a_{k}| + \sum\limits_{1< j< k \leq p} |f_{j}| |f_{k}| + |f_{1}| 
\sum\limits_{j = i}^{p} |a_{j}|}
\\
&(-1)^{\sum\limits_{1< j \leq k <i} |a_{k}| |f_{j}|+\sum\limits_{i< j \leq k <p} |a_{k}| |f_{j}| + \sum\limits_{j=i}^{p-1} \sum\limits_{k=i+1}^{p} |a_{j}| |f_{k}|}.
\end{split}
\end{equation}
\end{fact}
\begin{proof}
By \eqref{eq:mncor}, the right member of \eqref{eq:relacion-lemma-tecnicon} coincides with 
\begin{equation}
\label{eq:tec1}
{}^{i}\Box_{f_{1}, \dots, f_{p}}^{a_{1}, \dots, a_{p}} s_{f_{1}, \dots, f_{i}}^{a_{1}, \dots, a_{i-1},b_{i}}
(f_{1} \otimes \dots \otimes f_{i})\big(\lr{a_{1}, \dots, a_{i-1},b_{i}}_{i}\big)
\end{equation}
where $b_{i} = m_{2(p-i)+1}(a_{p},tf_{p}, \dots, a_{i+1},tf_{i+1},a_{i})$, 
whereas the left member of \eqref{eq:relacion-lemma-tecnicon} is by definition
\begin{equation}
\label{eq:tec2}
(-1)^{\sum\limits_{j=i+1}^{p}\sum\limits_{k=1}^{i}|f_{j}| |f_{k}|+\hskip -0.2cm\sum\limits_{i<j<k\leq p} \hskip -0.2cm |f_{j}| |f_{k}|}
(f_{1} \otimes \dots \otimes f_{i})\big(\lr{a_{1}, \dots, a_{i-1},c_{1}}_{i}\big)
\prod_{j=i+1}^{p} f_{j}(c_{j-i+1}),
\end{equation}
where $c_{1} \otimes \dots \otimes c_{p-i+1} = \lr{a_{i}, \dots, a_{p}}_{p-i+1}$. 
As a consequence, \eqref{eq:relacion-lemma-tecnicon} is tantamount to 
\begin{equation}
\label{eq:tec3}
\begin{split}
&{}^{i}\Box_{f_{1}, \dots, f_{p}}^{a_{1}, \dots, a_{p}} s_{f_{1}, \dots, f_{i}}^{a_{1}, \dots, a_{i-1},b_{i}}
\gan\big(m_{2(p-i)+1}(a_{p},tf_{p}, \dots, a_{i+1},tf_{i+1},a_{i}),tg_{i}\big) 
\\
&= (-1)^{z}
(g_{i} \otimes f_{i+1} \otimes \dots \otimes f_{p}) \big(\lr{a_{i}, \dots, a_{p}}_{p-i+1}\big),
\end{split}
\end{equation}
for all $g_{i} \in A^{\#}$ homogeneous of degree $i+1-p+|a_{i}|+\sum_{j=i+1}^{p}(|a_{j}| -|f_{j}|)$, 
where 
\[     z = \sum\limits_{j=i+1}^{p}\sum\limits_{k=1}^{i}|f_{j}| |f_{k}|+\big(i+1-p+|a_{i}|+\sum\limits_{j=i+1}^{p}(|a_{j}| -|f_{j}|)\big)\big(\sum\limits_{j=i+1}^{p}|f_{j}|\big).     \]
Using \eqref{eq:mncor} on the left member of \eqref{eq:tec3} we conclude that 
\[     {}^{i}\Box_{f_{1}, \dots, f_{p}}^{a_{1}, \dots, a_{p}} = (-1)^{z} s_{f_{1}, \dots, f_{i}}^{a_{1}, \dots, a_{i-1},b_{i}} 
s_{g_{i}, f_{i+1}, \dots, f_{p}}^{a_{i}, \dots, a_{p}}.     \] 
After using \eqref{eq:mncor2} in the previous identity and a lengthy but straightforward computation, the statement follows.
\end{proof}

Applying Fact \ref{fact:jacobi} to \eqref{eq:ainftyalgebraodd-odd3-1} we obtain that the first two lines in 
\eqref{eq:ainftyalgebraodd-odd2} give exactly 
\begin{equation}
\label{eq:ainftyalgebraodd-odd4-1}
\begin{split}
   &\sum_{i=1}^{p} \sum_{r=0}^{i-1} {}^{i}\Box_{f_{r}, \dots, f_{1},f_{p},\dots,f_{r+1}}^{a_{r}, \dots, a_{1},a_{p},\dots,a_{r+1}} (-1)^{\alpha} (f_{r} \otimes \dots \otimes f_{1} \otimes f_{p} \otimes\dots \otimes f_{r+1})
   \\
   &\big(\lr{a_{r}, \dots, a_{1}, a_{p}, \dots, a_{r+p-i+2},\lr{a_{r+p-i+1}, \dots, a_{i+1}}_{p-i+1}}_{i,L}\big),
\end{split}
\end{equation}
where $\alpha = (r+\sum_{j=1}^{r}(|a_{j}|+|f_{j}|))(p-r+\sum_{j=r+1}^{p}(|a_{j}|+|f_{j}|))$, 
whereas the same result applied to \eqref{eq:ainftyalgebraodd-odd3-2} tells us that the latter is precisely 
\begin{equation}
\label{eq:ainftyalgebraodd-odd4-2}
\begin{split}
   &\sum_{i=1}^{p} \sum_{r=0}^{i-2} {}^{i}\Box_{f_{r+p-i+1}, \dots, f_{1},f_{p},\dots,f_{r+p-i+2}}^{a_{r+p-i+1}, \dots, a_{1},a_{p},\dots,a_{r+p-i+2}} (f_{r+p-i+1} \otimes \dots \otimes f_{1} \otimes f_{p} \otimes\dots \otimes f_{r+p-i+2})
   \\
   &\big(\lr{a_{r+p-i+1}, \dots, a_{r+2}, \lr{ a_{r+1}, \dots, a_{1}, a_{p}, \dots, a_{r+p-i+2}}_{i} }_{p-i+1,L}\big) (-1)^{\beta},
\end{split}
\end{equation}
where $\beta = (r+p-i+1+\sum_{j=1}^{r+p-i+1}(|a_{j}|+|f_{j}|))(i+r+1+\sum_{j=r+p-i+2}^{p}(|a_{j}|+|f_{j}|))$. 

Let $\sigma \in \mathbb{S}_{p}$ be the unique cyclic permutation sending $1$ to $2$. 
Using \eqref{eq:unipermvecfun}, we see that \eqref{eq:ainftyalgebraodd-odd4-1} and \eqref{eq:ainftyalgebraodd-odd4-2} 
are equivalent to 
\begin{equation}
\label{eq:ainftyalgebraodd-odd5-1}
\begin{split}
   &\sum_{i=1}^{p} \sum_{r=0}^{i-1} {}^{i}\Box_{f_{r}, \dots, f_{1},f_{p},\dots,f_{r+1}}^{a_{r}, \dots, a_{1},a_{p},\dots,a_{r+1}} (-1)^{\alpha'} (f_{p} \otimes \dots \otimes f_{1})
   \\
   &\big(\sigma^{-r}\lr{a_{r}, \dots, a_{1}, a_{p}, \dots, a_{r+p-i+2},\lr{a_{r+p-i+1}, \dots, a_{i+1}}_{p-i+1}}_{i,L}\big),
\end{split}
\end{equation}
where 
\[     \alpha' = \big(r+\sum_{j=1}^{r}(|a_{j}|+|f_{j}|)\big)\big(p-r+\sum_{j=r+1}^{p}(|a_{j}|+|f_{j}|)\big) + \sum_{j=1}^{r}\sum_{k=r+1}^{p}|f_{j}||f_{k}|,     \]
and
\begin{equation}
\label{eq:ainftyalgebraodd-odd5-2}
\begin{split}
   &\sum_{i=1}^{p} \sum_{r=0}^{i-2} {}^{i}\Box_{f_{r+p-i+1}, \dots, f_{1},f_{p},\dots,f_{r+p-i+2}}^{a_{r+p-i+1}, \dots, a_{1},a_{p},\dots,a_{r+p-i+2}} (-1)^{\beta'} (f_{p} \otimes \dots \otimes f_{1})
   \\
   &\big(\sigma^{-(r+p-i+1)} \lr{a_{r+p-i+1}, \dots, a_{r+2}, \lr{ a_{r+1}, \dots, a_{1}, a_{p}, \dots, a_{r+p-i+2}}_{i} }_{p-i+1,L}\big),
\end{split}
\end{equation}
where 
\begin{equation*}
\begin{split}
     \beta' &= \big(r+p-i+1+\sum_{j=1}^{r+p-i+1}(|a_{j}|+|f_{j}|)\big)\big(i+r+1
     \\ 
     &+ \sum_{j=r+p-i+2}^{p}(|a_{j}|+|f_{j}|)\big)  + \sum_{j=1}^{r+p-i+1}\sum_{k=r+p-i+2}^{p}|f_{j}||f_{k}|,      
\end{split}
\end{equation*}
respectively.
Furthermore, if we reindex \eqref{eq:ainftyalgebraodd-odd5-2} by setting $i'=p-i+1$ and $r'=r+p-i+1$, the former becomes 
\begin{equation}
\label{eq:ainftyalgebraodd-odd5-2b}
\begin{split}
   &\sum_{i'=1}^{p} \sum_{r'=i'}^{p-1} {}^{i}\Box_{f_{r'}, \dots, f_{1},f_{p},\dots,f_{r'+1}}^{a_{r'}, \dots, a_{1},a_{p},\dots,a_{r'+1}} (-1)^{\beta''} (f_{p} \otimes \dots \otimes f_{1})
   \\
   &\big(\sigma^{-r'} \lr{a_{r'}, \dots, a_{r'-i'+2}, \lr{ a_{r'-i'+1}, \dots, a_{1}, a_{p}, \dots, a_{r'+1}}_{p-i'+1} }_{i',L}\big),
\end{split}
\end{equation}
where 
\[     \beta'' = \big(r'+\sum_{j=1}^{r'}(|a_{j}|+|f_{j}|)\big)\big(r'+p+\sum_{j=r'+1}^{p}(|a_{j}|+|f_{j}|)\big) + \sum_{j=1}^{r'}\sum_{k=r'+1}^{p} |f_{j}||f_{k}|.     \] 

On the other hand, after a tedious but straightforward calculation, we see that, for all $i \in \{ 1, \dots, p\}$ and $r \in \{ 0, \dots, i-1\}$,  
\begin{equation}
\label{eq:ainftyalgebraodd-odd6-1}
\begin{split}
   {}^{i}&\Box_{f_{r}, \dots, f_{1},f_{p},\dots,f_{r+1}}^{a_{r}, \dots, a_{1},a_{p},\dots,a_{r+1}} (-1)^{\alpha''}
   \\
   &= (-1)^{(p+1)(i+r) + \sum\limits_{j=1}^{r} (i-j) |a_{j}| + \sum\limits_{j=r+1}^{r+p-i+1} (j-1) |a_{j}| + \sum\limits_{j=r+p-i+2}^{p} (i-j) |a_{j}| } 
   \\
   &\hskip 0.48cm (-1)^{\sum\limits_{j=1}^{p} (j-1) |f_{j}|+\sum\limits_{1\leq j< k \leq p} |a_{j}| |a_{k}|+\sum\limits_{j=1}^{r+p-i+1}\sum\limits_{k=r+p-i+2}^{p} |a_{j}| |a_{k}|}
\\
&\hskip 0.48cm (-1)^{\sum\limits_{1\leq j< k \leq p} |f_{j}| |f_{k}| + \sum\limits_{1\leq j \leq k \leq p} |a_{j}| |f_{k}|+\sum\limits_{j=r+1}^{r+p-i+1} \sum\limits_{k=1}^{p} |a_{j}| |f_{k}|},
\end{split}
\end{equation}  
and for all $i \in \{ 1, \dots, p\}$ and $r \in \{ i, \dots, p-1\}$
\begin{equation}
\label{eq:ainftyalgebraodd-odd6-2}
\begin{split}
   {}^{i}&\Box_{f_{r}, \dots, f_{1},f_{p},\dots,f_{r+1}}^{a_{r}, \dots, a_{1},a_{p},\dots,a_{r+1}} (-1)^{\beta''}
   \\
   &= (-1)^{(p+1)(i+r) + \sum\limits_{j=1}^{r-i+1} (j-1) |a_{j}| + \sum\limits_{j=r-i+2}^{r} (i-j) |a_{j}| + \sum\limits_{j=r+1}^{p} (j-1) |a_{j}| } 
   \\
   &\hskip 0.48cm (-1)^{\sum\limits_{j=1}^{p} (j-1) |f_{j}|+\sum\limits_{1\leq j< k \leq p} |a_{j}| |a_{k}|+\sum\limits_{j=1}^{r-i+1}\sum\limits_{k=r-i+2}^{p} |a_{j}| |a_{k}|}
\\
&\hskip 0.48cm (-1)^{\sum\limits_{1\leq j< k \leq p} |f_{j}| |f_{k}| + \sum\limits_{1\leq j \leq k \leq p} |a_{j}| |f_{k}|+(\sum\limits_{j=1}^{r-i+1} |a_{j}| + \sum\limits_{j=r+1}^{p} |a_{j}|)(\sum\limits_{k=1}^{p} |f_{k}|)}.
\end{split}
\end{equation}  
The Koszul sign rule tells us that, for $i \in \{ 1, \dots, p\}$ and $r \in \{ 0, \dots, i-1\}$, 
\begin{equation}
\label{eq:ainftyalgebraodd-odd7-1}
\begin{split}
&\lr{a_{r}, \dots, a_{1}, a_{p}, \dots, a_{r+p-i+2},\lr{a_{r+p-i+1}, \dots, a_{i+1}}_{p-i+1}}_{i,L}
\\ 
&= (-1)^{\bar{\alpha}} \big(\lr{\dots}_{i,p-i+1} \circ \sigma^{r}\big) (a_{p} \otimes \dots \otimes a_{1}),
\end{split}
\end{equation} 
where
\[     \bar{\alpha} = (p-i-1)\big(\sum\limits_{j=1}^{r} |a_{j}| + \sum\limits_{j=r+p-i+2}^{p} |a_{j}|\big) + \sum\limits_{j=1}^{r}\sum\limits_{k=r+1}^{p} |a_{j}| |a_{k}|,     \]
whereas, for $i \in \{ 1, \dots, p\}$ and $r \in \{ i, \dots, p-1\}$, we have that 
\begin{equation}
\label{eq:ainftyalgebraodd-odd7-2}
\begin{split}
&\lr{a_{r}, \dots, a_{r-i+2}, \lr{ a_{r-i+1}, \dots, a_{1}, a_{p}, \dots, a_{r+1}}_{p-i+1} }_{i,L}
\\ 
&= (-1)^{\bar{\beta}} \big(\lr{\dots}_{i,p-i+1} \circ \sigma^{r}\big) (a_{p} \otimes \dots \otimes a_{1}),
\end{split}
\end{equation} 
where
\[     \bar{\beta} = (p-i-1)\big(\sum\limits_{j=r-i+2}^{r} |a_{j}|\big) + \sum\limits_{j=1}^{r}\sum\limits_{k=r+1}^{p} |a_{j}| |a_{k}|.     \]
Using $\sum_{j=1}^{p}|f_{j}| = p-1+  \sum_{j=1}^{p}|a_{j}|$ ($\operatorname{mod} 2$) in the last term of the right member of 
\eqref{eq:ainftyalgebraodd-odd6-1} and utilizing this result together with \eqref{eq:ainftyalgebraodd-odd7-1} in \eqref{eq:ainftyalgebraodd-odd5-1}, we see that the latter is equivalent to
\begin{equation}
\label{eq:ainftyalgebraodd-odd8-1}
   \sum_{i=1}^{p} \sum_{r=0}^{i-1} (-1)^{(p+1)(i+r)+\hat{\alpha}} (f_{p} \otimes \dots \otimes f_{1})
\Big(\big(\sigma^{-r}\circ \lr{\dots}_{i,p-i+1} \circ \sigma^{r}\big) (a_{p} \otimes \dots \otimes a_{1})\Big),
\end{equation}
where $\hat{\alpha}$ is given by 
\[      \sum\limits_{j=1}^{p} (p-j-1) |a_{j}|+ \sum\limits_{j=1}^{p} (j-1) |f_{j}|+\sum\limits_{1\leq j< k \leq p} |a_{j}| |a_{k}|+ \sum\limits_{1\leq j< k \leq p} |f_{j}| |f_{k}| + \sum\limits_{1\leq j \leq k \leq p} |a_{j}| |f_{k}|.     \]
The precise same argument but involving instead \eqref{eq:ainftyalgebraodd-odd7-2} and \eqref{eq:ainftyalgebraodd-odd6-2} in \eqref{eq:ainftyalgebraodd-odd5-2b} yields that the latter is tantamount to 
\begin{equation}
\label{eq:ainftyalgebraodd-odd8-2}
   \sum_{i=1}^{p} \sum_{r=i}^{p-1} (-1)^{(p+1)(i+r)+\hat{\alpha}} (f_{p} \otimes \dots \otimes f_{1})
\Big(\big(\sigma^{-r}\circ \lr{\dots}_{i,p-i+1} \circ \sigma^{r}\big) (a_{p} \otimes \dots \otimes a_{1})\Big).
\end{equation}
As a consequence, \eqref{eq:ainftyalgebraodd-odd2} is exactly 
\begin{equation}
\label{eq:ainftyalgebraodd-odd9}
   (-1)^{\hat{\alpha}} \sum_{i=1}^{p} \sum_{r=i}^{p-1} (-1)^{(p+1)(i+r)} (f_{p} \otimes \dots \otimes f_{1})
\Big(\big(\sigma^{-r}\circ \lr{\dots}_{i,p-i+1} \circ \sigma^{r}\big) (a_{p} \otimes \dots \otimes a_{1})\Big).
\end{equation}
Since $\operatorname{sgn}(\sigma) = (-1)^{p+1}$, we obtain precisely \eqref{eq:dpoisinfjac}, as was to be shown. 

We will show that \eqref{eq:bij2} is bijective. 
Note first that, given any good and manageable $d$-pre-Calabi-Yau structure $\{m_{\bullet}\}_{\bullet \in \NN}$ on $A$, it is uniquely determined by $m_{2q+1}|_{(A \otimes A^{\#}[-1])^{\otimes q} \otimes A}$, for all $q \in \NN_{0}$. 
Indeed, the fact that the pre-Calabi-Yau structure on $A$ is good tells us that the full $m_{2q+1}$ on $\partial_{-1}A$ 
is unique, and the manageability hypothesis implies that $m_{2}$ is uniquely determined by the 
algebra structure of $A$. 
As a consequence, and using that the identity \eqref{eq:mncor} implies that the corresponding double bracket $\lr{\dots}_{q+1}$ completely determines $m_{2q+1}|_{(A \otimes A^{\#}[-1])^{\otimes q} \otimes A}$, we conclude that \eqref{eq:bij2} is injective. 

We will finally show that \eqref{eq:bij2} is surjective. 
It suffices to prove that, given any collection of good morphisms $m_{2q+1} : \partial_{-1}A^{\otimes (2q+1)} \rightarrow \partial_{-1}A$ of degree $1-2q$ for $q \in \NN_{0}$ on the graded algebra $\partial_{-1}A$, whose product is denoted by $m_{2}$, satisfying the cyclic identities \eqref{eq:ip1}, for the natural bilinear form $\gan$ of degree $- 1$, then the vanishing of 
$\operatorname{SI}(2p)_{\gan}|_{A\otimes (A \otimes A^{\#}[-1])^{\otimes p}}$ is equivalent to 
$\operatorname{SI}(2p)_{\gan} = 0$, and 
$\operatorname{SI}(2p-1)_{\gan}|_{(A \otimes A^{\#}[-1])^{\otimes p}}=0$ 
is tantamount to the vanishing of $\operatorname{SI}(2p-1)_{\gan}$, for all $p \in \NN$. 
We leave to the reader the tedious but straightforward verification that the vanishing of $\operatorname{SI}(2p)_{\gan}|_{A\otimes (A \otimes A^{\#}[-1])^{\otimes p}}$ and that of $\operatorname{SI}(2p)_{\gan}|_{\sigma(A\otimes (A \otimes A^{\#}[-1])^{\otimes p})}$ are equivalent, for any $\sigma \in C_{2p+1}$, whereas $\operatorname{SI}(2p)_{\gan}|_{\sigma(A\otimes (A \otimes A^{\#}[-1])^{\otimes p})}$ trivially vanishes if $\sigma \in \mathbb{S}_{2p+1} \setminus C_{2p+1}$.
Similarly, it is long but easy to verify that 
$\operatorname{SI}(2p-1)_{\gan}|_{(A \otimes A^{\#}[-1])^{\otimes p}} = 0$ is equivalent to $\operatorname{SI}(2p-1)_{\gan}|_{\sigma((A \otimes A^{\#}[-1])^{\otimes p})} = 0$, for any cyclic permutation $\sigma \in C_{2p} \subseteq \mathbb{S}_{2p}$, and $\operatorname{SI}(2p-1)_{\gan}|_{\sigma((A \otimes A^{\#}[-1])^{\otimes p})}$ is trivially zero if $\sigma \in \mathbb{S}_{2p} \setminus C_{2p}$.
This concludes the proof of the theorem. 
\end{proof}

\bibliographystyle{model1-num-names}
\addcontentsline{toc}{section}{References}



\vskip 1cm
\noindent \scshape{David Fern\'andez: Fakult\"at f\"ur Mathematik, Universit\"at Bielefeld, 33501 Bielefeld, Germany.}
\\
\normalfont\textit{E-mail address}: \href{mailto:david.fernandez@math.uni-bielefeld.de}{\texttt{david.fernandez@math.uni-bielefeld.de}}.

\medskip
\noindent \scshape{Estanislao Herscovich: Institut Fourier, Universit\'e Grenoble Alpes, 38610 Gi\`eres, France.}
\\
\normalfont\textit{E-mail address}: \href{mailto:Estanislao.Herscovich@univ-grenoble-alpes.fr}{\texttt{Estanislao.Herscovich@univ-grenoble-alpes.fr}}.

\end{document}